\documentclass[12pt]{article}
\usepackage{BOONDOX-cal}
\usepackage{amssymb,a4}
\usepackage{amsmath,amsfonts,amssymb,amsthm, mathrsfs}

\allowdisplaybreaks[4]
\newtheorem{thm}{Theorem}[section]
\newtheorem{lem}[thm]{Lemma}
\newtheorem{defi}[thm]{Definition}
\newtheorem{prop}[thm]{Proposition}
\newtheorem{rem}[thm]{Remark}
\newtheorem{cor}[thm]{Corollary}

\begin{document}
\begin{center}
{\Large \bf Affine vertex operator superalgebra $L_{\widehat{osp(1|2)}}(\mathcal{l},0)$ at admissible level}
\end{center}

\begin{center}
{Huaimin Li, Qing Wang
		\footnote{Supported by China NSF grants No.12071385 and No.12161141001.}\\
School of Mathematical Sciences, Xiamen University, Xiamen, China 361005}	
\end{center}

\begin{abstract}
Let $L_{\widehat{osp(1|2)}}(\mathcal{l},0)$ be the simple affine vertex operator superalgebra with admissible level $\mathcal{l}$.
We prove that the category of weak $L_{\widehat{osp(1|2)}}(\mathcal{l},0)$-modules
on which the positive part of $\widehat{osp(1|2)}$ acts locally nilpotent is semisimple.
Then we prove that $\mathbb{Q}$-graded vertex operator superalgebras $(L_{\widehat{osp(1|2)}}(\mathcal{l},0),\omega_\xi)$
with new Virasoro elements $\omega_\xi$ are rational
and the irreducible modules are exactly the admissible modules for $\widehat{osp(1|2)}$,
where $0<\xi<1$ is a rational number.
Furthermore, we determine the Zhu's algebras $A(L_{\widehat{osp(1|2)}}(\mathcal{l},0))$
and their bimodules $A(L(\mathcal{l},\mathcal{j}))$ for $(L_{\widehat{osp(1|2)}}(\mathcal{l},0),\omega_\xi)$,
where $\mathcal{j}$ is the admissible weight.
As an application, we calculate the fusion rules among the irreducible ordinary modules of $(L_{\widehat{osp(1|2)}}(\mathcal{l},0),\omega_\xi)$.
\end{abstract}

\section{Introduction}

Kac and Wakimoto introduced the notion of admissible highest weight representations for affine Lie (super)algebras to study the modular invariant representations of affine Lie (super)algebras in \cite{KW},
then they classified the admissible weights for all affine Lie algebras in \cite{KW2}.
Admissible modules also studied in the context of vertex operator algebras, in \cite{AM},
Adamovi\'{c} and Milas proved the simple affine vertex operator algebra
$L_{\widehat{sl_2}}(\mathcal{l},0)$ at admissible level is rational in the category $\mathcal{O}$,
and conjectured this holds for simple affine vertex operator algebras $L_{\widehat{\mathfrak{g}}}(\mathcal{l},0)$
associated to any simple finite-dimensional Lie algebra $\mathfrak{g}$.
This conjecture has been proved by Arakawa in \cite{A}.
However, simple affine vertex operator algebras $L_{\widehat{\mathfrak{g}}}(\mathcal{l},0)$
at admissible level $\mathcal{l}$ may not be rational if $\mathcal{l}$ is not a positive integer.
In \cite{DLM2}, Dong, Li and Mason showed that $L_{\widehat{sl_2}}(\mathcal{l},0)$ is a
rational $\mathbb{Q}$-graded vertex operator algebra
under a new Virasoro element and that the irreducible modules are exactly the admissible modules for $\widehat{sl_2}$.
Furthermore, these results has been proved by Lin for
$L_{\widehat{\mathfrak{g}}}(\mathcal{l},0)$ associated to any simple finite-dimensional Lie algebra $\mathfrak{g}$ in \cite{Lin}.
For superalgebra case,
Wood classified the irreducible relaxed highest-weight modules for the simple affine vertex operator superalgebra $L_{\widehat{osp(1|2)}}(\mathcal{l},0)$ at admissible level,
and also proved that $L_{\widehat{osp(1|2)}}(\mathcal{l},0)$ is rational in the category $\mathcal{O}$ in \cite{W}.
Gorelik and Serganova later
showed the rationality in the category $\mathcal{O}$ for
$L_{\widehat{osp(1|2n)}}(\mathcal{l},0)$ in \cite{GS}.

Let $\mathfrak{g}={osp}(1|2)=\mbox{span}_\mathbb{C}\{h,e,f,x,y\}$,
where $\{h,e,f,x,y\}$ is the standard Chevalley basis of $\mathfrak{g}$.
In our paper we study the simple affine vertex operator superalgebra $L_{\widehat{\mathfrak{g}}}(\mathcal{l},0)$ with admissible level $\mathcal{l}$.
Let $V_{\widehat{\mathfrak{g}}}(\mathcal{l},\mathbb{C})$ be the universal affine vertex operator superalgebra,
for a $\mathbb{C}h$-module $U$,
we construct the weak $V_{\widehat{\mathfrak{g}}}(\mathcal{l},\mathbb{C})$-module $L(\mathcal{l},U)$ in Subsection \ref{sec:3.1}
and determine the condition that
$L(\mathcal{l},U)$ is a weak $L_{\widehat{\mathfrak{g}}}(\mathcal{l},0)$-module.
We prove that the irreducible highest weight $\widehat{\mathfrak{g}}$-module
$L(\mathcal{l},\mathcal{j})$ is a weak $L_{\widehat{\mathfrak{g}}}(\mathcal{l},0)$-module if and only if $\mathcal{j}$ is admissible,
this result implies that the weak $L_{\widehat{\mathfrak{g}}}(\mathcal{l},0)$-modules in category $\mathcal{O}$ are complete reducible
and the irreducible modules are those $L(\mathcal{l},\mathcal{j})$ for $\mathcal{j}$ being admissible.
This also obtained in \cite{W} by computing explicit presentation for the Zhu's algebra associated to $L_{\widehat{\mathfrak{g}}}(\mathcal{l},0)$
and in \cite{GS} by establishing the correspondence between the weak $L_{\widehat{osp(1|2n)}}(\mathcal{l},0)$-modules
and the weak $L_{\widehat{sp_{2n}}}(\mathcal{l},0)$-modules.
Furthermore, we consider the category $\mathcal{C}_{\mathcal{l}}$ in Subsection \ref{sec:3.2},
where $\mathcal{C}_{\mathcal{l}}$ is the subcategory of the weak $L_{\widehat{\mathfrak{g}}}(\mathcal{l},0)$-module category
such that $M$ is an object in $\mathcal{C}_{\mathcal{l}}$ if and only if the sum of all positive root spaces of
$\widehat{\mathfrak{g}}$ acts locally nilpotently on $M$.
We prove that the category $\mathcal{C}_{\mathcal{l}}$ is semisimple
and there are finitely many irreducible weak
$L_{\widehat{\mathfrak{g}}}(\mathcal{l},0)$-modules belonging to $\mathcal{C}_{\mathcal{l}}$ up to isomorphism.
It implies that any ordinary
$L_{\widehat{\mathfrak{g}}}(\mathcal{l},0)$-module is completely reducible and
there are finitely many irreducible ordinary $L_{\widehat{\mathfrak{g}}}(\mathcal{l},0)$-modules up to isomorphism,
this result also established for $\mathfrak{g}=osp(1|2n)$ in \cite{CGL} by using the theory of vertex superalgebra extensions.
Moreover, we consider $\mathbb{Q}$-graded vertex operator superalgebras $(L_{\widehat{\mathfrak{g}}}(\mathcal{l},0),\omega_\xi)$
associated to a family of new Virasoro elements $\omega_\xi$ in Subsection \ref{sec:3.3}, where $0<\xi<1$ is a rational number.
We apply the semisimplicity of category $\mathcal{C}_{\mathcal{l}}$ to prove that
$(L_{\widehat{\mathfrak{g}}}(\mathcal{l},0),\omega_\xi)$
are rational $\mathbb{Q}$-graded vertex operator superalgebras.
In \cite{DK}, the $A(V)$-theory for $\mathbb{Q}$-graded vertex operator superalgebras has been studied.
In this paper we determine the Zhu's algebras $A(L_{\widehat{\mathfrak{g}}}(\mathcal{l},0))$ and their bimodules $A(L(\mathcal{l},\mathcal{j}))$
for $(L_{\widehat{\mathfrak{g}}}(\mathcal{l},0),\omega_\xi)$ in Subsection \ref{sec:4.2}. Then
we apply the Zhu's algebras and their bimodules associated to $(L_{\widehat{\mathfrak{g}}}(\mathcal{l},0),\omega_\xi)$
to calculate the fusion rules among the irreducible ordinary modules of $(L_{\widehat{\mathfrak{g}}}(\mathcal{l},0),\omega_\xi)$,
the fusion rules we obtained are also those with respect to the old vertex operator superalgebra $L_{\widehat{\mathfrak{g}}}(\mathcal{l},0)$.
The fusion rules among the irreducible relaxed highest-weight modules over $L_{\widehat{\mathfrak{g}}}(\mathcal{l},0)$
have been calculated in \cite{CKLR}
by viewing $L_{\widehat{\mathfrak{g}}}(\mathcal{l},0)$ as extensions of the tensor product of certain
simple Virasoro vertex operator algebra and simple affine vertex operator algebra associated to $sl_2$.

This paper is organized as follows.
In Section \ref{sec:2}, we recall some concepts about $\mathbb{Z}$-graded vertex operator superalgebras and
some facts about admissible modules of $\widetilde{\mathfrak{g}}$.
In Section \ref{sec:3}, we construct the weak $V_{\widehat{\mathfrak{g}}}(\mathcal{l},\mathbb{C})$-module $L(\mathcal{l},U)$
and give the condition for $L(\mathcal{l},U)$ being a weak $L_{\widehat{\mathfrak{g}}}(\mathcal{l},0)$-module.
Then we prove that $L_{\widehat{\mathfrak{g}}}(\mathcal{l},0)$ at admissible level is rational
in category $\mathcal{C}_{\mathcal{l}}$ and $(L_{\widehat{\mathfrak{g}}}(\mathcal{l},0),\omega_\xi)$
are rational $\mathbb{Q}$-graded vertex operator superalgebras with respect to a family of Virasoro elements $\omega_\xi$.
In Section \ref{sec:4}, we determine the Zhu's algebra $A(L_{\widehat{\mathfrak{g}}}(\mathcal{l},0))$ and their bimodules $A(L(\mathcal{l},\mathcal{j}))$
for $(L_{\widehat{\mathfrak{g}}}(\mathcal{l},0),\omega_\xi)$.
As an application, we calculate the fusion rules among the irreducible ordinary modules of $(L_{\widehat{\mathfrak{g}}}(\mathcal{l},0),\omega_\xi)$.
Throughout the paper,
 $\mathbb{Z}$, $\mathbb{N}$, $\mathbb{Z}_+$, $\mathbb{Q}$ and $\mathbb{C}$ are the sets of integers,
nonnegative integers, positive integers, rational numbers and complex numbers respectively.
\section{Preliminaries}
\label{sec:2}
	\def\theequation{2.\arabic{equation}}
	\setcounter{equation}{0}
In this section, we recall some concepts about $\mathbb{Z}$-graded vertex operator superalgebra and
some facts about admissible modules of $\widetilde{{osp}(1|2)}$.
\subsection{Vertex operator superalgebras and their modules}

Let $V=V_{\bar{0}}\oplus V_{\bar{1}}$ be a vector superspace,
the elements in $V_{\bar{0}}$ (resp. $V_{\bar{1}}$) are called {\em even} (resp. {\em odd}).
For any $v\in V_{\bar{i}}$ with $i=0,1$, define $|v|=i$.
First we recall the definitions of $\mathbb{Z}$-graded vertex operator superalgebra
and their various module categories following \cite{DZ2,L}.

\begin{defi}\label{defivosa}
{\em A {\em vertex superalgebra} is a quadruple $(V,\textbf{1},D,Y)$,
where $V=V_{\bar{0}}\oplus V_{\bar{1}}$ is a vector superspace,
$D$ is an endomorphism of $V$, $\textbf{1}$ is a specified even vector called the {\em vacuum vector} of $V$,
and $Y$ is a linear map
\begin{equation*}
\begin{aligned}
Y(\cdot,z):&V\rightarrow (\mbox{End} V)[[z,z^{-1}]]\\
&a\mapsto Y(a,z)=\sum_{n\in\mathbb{Z}}a_n z^{-n-1} ~~(a_n\in\mbox{End} V)
\end{aligned}
\end{equation*}
such that\\
(1) For any $a,b\in V$, $a_nb=0$ for $n$ sufficiently large;\\
(2) $[D, Y(a,z)]=Y(D(a),z)=\frac{d}{dz}Y(a,z)$ for any $a\in V$;\\
(3) $Y(\textbf{1},z)=\mbox{Id}_V$;\\
(4) $Y(a,z)\textbf{1}\in (\mbox{End} V)[[z]]$ and $\mbox{lim}_{z\rightarrow0}Y(a,z)\textbf{1}=a$ for any $a\in V$;\\
(5) For $\mathbb{Z}_2$-homogeneous elements $a,b\in V$, the following {\em Jacobi identity} holds:
\begin{equation*}
\begin{aligned}
  &z_0^{-1}\delta(\frac{z_1-z_2}{z_0})Y(a,z_1)Y(b,z_2)-(-1)^{|a||b|}z_0^{-1}\delta(\frac{z_2-z_1}{-z_0})Y(b,z_2)Y(a,z_1)\\
  &=z_2^{-1}\delta(\frac{z_1-z_0}{z_2})Y(Y(a,z_0)b,z_2).
 \end{aligned}
\end{equation*}

A vertex superalgebra $V$ is called a {\em $\mathbb{Z}$-graded vertex operator superalgebra}
if there is an even vector $\omega$ called the {\em Virasoro element} of $V$ such that\\
(6) Set $Y(\omega,z)=\sum_{n\in\mathbb{Z}}L(n) z^{-n-2}$, for any $m,n\in\mathbb{Z}$,
$$[L(m), L(n)]=(m-n)L(m+n)+\frac{m^{3}-m}{12}\delta_{m+n,0}c,$$
where $c\in\mathbb{C}$ is called the {\em central charge};\\
(7) $L(-1)=D$;\\
(8) $V$ is $\mathbb{Z}$-graded such that $V=\bigoplus_{n\in\mathbb{Z}}V_{(n)}$,
$L(0)\mid_{V_{(n)}}=n \mbox{Id}_{V_{(n)}}$, $\mbox{dim}~V_{(n)}<\infty$ for all $n\in\mathbb{Z}$ and
$V_{(n)}=0$ for $n$ sufficiently small. For $v\in V_{(n)}$,
the {\em conformal weight} of $v$ is defined to be $\mbox{wt}~v=n$.}
\end{defi}

\begin{defi}
{\em  Let $V$ be a $\mathbb{Z}$-graded vertex operator superalgebra.
A {\em weak $V$-module} is a triple $(M,d,Y_M)$ where
$M=M_{\bar{0}}\oplus M_{\bar{1}}$ is a vector supersapce,
$d$ is an endomorphism of $M$
and $Y_M$ is a linear map
\begin{equation*}
\begin{aligned}
Y_M(\cdot,z):&V\rightarrow (\mbox{End} M)[[z,z^{-1}]]\\
&a\mapsto Y_M(a,z)=\sum_{n\in\mathbb{Z}}a_n z^{-n-1} ~~(a_n\in\mbox{End} M)
\end{aligned}
\end{equation*}
such that\\
(1) For any $a\in V,u\in M$, $a_nu=0$ for $n$ sufficiently large;\\
(2) $[d, Y_M(a,z)]=Y_M(D(a),z)=\frac{d}{dz}Y_M(a,z)$ for any $a\in V$;\\
(3) $Y_M(\textbf{1},z)=\mbox{Id}_M$;\\
(4) For $\mathbb{Z}_2$-homogeneous elements $a,b\in V$, the following {\em Jacobi identity} holds:
\begin{equation*}
\begin{aligned}
  &z_0^{-1}\delta(\frac{z_1-z_2}{z_0})Y_M(a,z_1)Y_M(b,z_2)-(-1)^{|a||b|}z_0^{-1}\delta(\frac{z_2-z_1}{-z_0})Y_M(b,z_2)Y_M(a,z_1)\\
  &=z_2^{-1}\delta(\frac{z_1-z_0}{z_2})Y_M(Y(a,z_0)b,z_2).
 \end{aligned}
\end{equation*}

A weak $V$-module $M$ is called an {\em admissible module} ($\mathbb{Z}_+$-graded weak module)
if $M$ has a $\mathbb{Z}_+$-gradation $M=\bigoplus_{n\in\mathbb{Z}_+}M(n)$ such that
\begin{equation*}
  a_mM(n)\subseteq M(\mbox{wt}~a+n-m-1)
\end{equation*}
for any $\mathbb{Z}$-homogeneous element $a\in V, m\in\mathbb{Z}, n\in\mathbb{Z}_+$.

A weak $V$-module $M$ is called an {\em ordinary module}
if $M=\bigoplus_{\lambda\in\mathbb{C}}M_{\lambda}$, where $M_{\lambda}=\{w\in M \mid L(0)w=\lambda w\}$ such that
$\mbox{dim}~M_{\lambda}<\infty$ for all $\lambda\in\mathbb{C}$ and
$M_{\lambda}=0$ for the real part of $\lambda$ sufficiently small.}
\end{defi}

\begin{defi}
{\em Let $V$ be a $\mathbb{Z}$-graded vertex operator superalgebra and $(W_i,Y_i),(W_j,Y_j)$, $(W_k,Y_k)$ be three weak $V$-modules.
An {\em intertwining operator of type $\begin{pmatrix}W_i\\ W_j~~W_k\end{pmatrix}$} is a linear map
\begin{equation*}
\begin{aligned}
\mathcal{Y}(\cdot,z):&W_j\rightarrow (\mbox{Hom} (W_k,W_i))\{z\}\\
&w\mapsto \mathcal{Y}(w,z)=\sum_{n\in\mathbb{Q}}w_n z^{-n-1} ~~(w_n\in\mbox{Hom} (W_k,W_i))
\end{aligned}
\end{equation*}
such that for any $v\in V, w^j\in W_j,w^k\in W_k$, \\
(1) $w_n^jw^k=0$ for $n$ sufficiently large;\\
(2) $\mathcal{Y}(L(-1)w^j,z)=\frac{d}{dz}\mathcal{Y}(w^j,z)$, where $L(-1)$ is the operator acting on $W_j$;\\
(3) for $\mathbb{Z}_2$-homogeneous elements $v, w^j$, the following {\em Jacobi identity} holds for the operators acting on the element $w^k$:
\begin{equation*}
\begin{aligned}
  &z_0^{-1}\delta(\frac{z_1-z_2}{z_0})Y_i(v,z_1)\mathcal{Y}(w^j,z_2)w^k-(-1)^{|v||w^j|}z_0^{-1}\delta(\frac{z_2-z_1}{-z_0})\mathcal{Y}(w^j,z_2)Y_k(v,z_1)w^k\\
  &=z_2^{-1}\delta(\frac{z_1-z_0}{z_2})\mathcal{Y}(Y_j(v,z_0)w^j,z_2)w^k.
 \end{aligned}
\end{equation*}
For any irreducible weak $V$-modules $W_i,W_j,W_k$,
the space of all intertwining operators of type $\begin{pmatrix}W_i\\ W_j~~W_k\end{pmatrix}$ is denoted by
$I_V\begin{pmatrix}W_i\\ W_j~~W_k\end{pmatrix}$.
Let $N_{W_j,W_k}^{W_i}=\mbox{dim}~I_V\begin{pmatrix}W_i\\ W_j~~W_k\end{pmatrix}$.
These integers $N_{W_j,W_k}^{W_i}$ are usually called the fusion rules.}
\end{defi}

\subsection{Admissible modules of $\widetilde{osp(1|2)}$}
\label{sec:2.2}

  Let $\mathfrak{g}={osp}(1|2)$ be the finite dimensional simple Lie superalgebra with the basis $\{h,e,f,x,y\}$
such that the even part $\mathfrak{g}_0=\mbox{span}_{\mathbb{C}}\{h,e,f\}\cong sl_2$ and
the odd part $\mathfrak{g}_1=\mbox{span}_{\mathbb{C}}\{x,y\}$.
The anticommutation relations are given by
\begin{equation}\nonumber
\begin{aligned}
&[e,f]=h,~[h,e]=2e,~[h,f]=-2f,\\
&[h,x]=x,~[e,x]=0,~[f,x]=-y,\\
&[h,y]=-y,~[e,y]=-x,~[f,y]=0,\\
&\{x,x\}=2e,\{x,y\}=h,\{y,y\}=-2f.
\end{aligned}
\end{equation}
The invariant nondegenerate even supersymmetric bilinear form $(\cdot,\cdot)$ on $\mathfrak{g}$ such that non-trivial products are given by
\begin{equation}\nonumber
(e,f)=(f,e)=1,(h,h)=2,(x,y)=-(y,x)=2.
\end{equation}

Let $\tilde{\mathfrak{g}}=\mathfrak{g}\otimes\mathbb{C}[t,t^{-1}]\oplus\mathbb{C}k$ be the affine Lie superalgebra of $osp(1|2)$
with the $\mathbb{Z}_2$-gradation
\begin{equation}\nonumber
|a\otimes t^n|=|a|~(a\in\mathfrak{g}),~|k|=0,
\end{equation}
and anticommutation relations
\begin{equation}\nonumber
[a\otimes t^n,b\otimes t^m]=[a,b]\otimes t^{m+n}+m\delta_{m+n,0}(a,b)k~(a,b\in\mathfrak{g},m,n\in\mathbb{Z}),~[k,\tilde{\mathfrak{g}}]=0.
\end{equation}
We identify $\mathfrak{g}$ with $\mathfrak{g}\otimes t^0$ and
set $a(n)=a\otimes t^n$ for $a\in\mathfrak{g}$ and $n\in\mathbb{Z}$ for convenience.
Define subalgebras
\begin{equation}
\begin{aligned}
&N_+=\mathbb{C}e\oplus\mathbb{C}x\oplus\mathfrak{g}\otimes t\mathbb{C}[t],\\
&N_-=\mathbb{C}f\oplus\mathbb{C}y\oplus\mathfrak{g}\otimes t^{-1}\mathbb{C}[t^{-1}],\\
&B=N_+\oplus\mathbb{C}h\oplus\mathbb{C}k,~P=\mathfrak{g}\otimes\mathbb{C}[t]\oplus\mathbb{C}k.
\end{aligned}
\end{equation}

Let $\hat{\mathfrak{g}}=\tilde{\mathfrak{g}}\oplus\mathbb{C}d$ be the extended affine Lie superalgebra with
$|d|=0$ and
\begin{equation}\nonumber
[d,a\otimes t^n]=na\otimes t^n~(a\in\mathfrak{g},n\in\mathbb{Z}),~[d,k]=0.
\end{equation}
Let $H=\mathbb{C}h\oplus\mathbb{C}k\oplus\mathbb{C}d$ be the Cartan subalgebra of $\hat{\mathfrak{g}}$.
For any $\lambda\in H^*$,
denote by $M(\lambda)$(resp. $L(\lambda)$) the Verma (resp. the irreducible highest weight) $\hat{\mathfrak{g}}$-module.
It is clear that $L(\lambda)$ is an irreducible $\tilde{\mathfrak{g}}$-module,
denote by $L(\mathcal{l},\mathcal{j})$ the $\tilde{\mathfrak{g}}$-module $L(\lambda)$, where $\mathcal{l}=\langle\lambda,k\rangle,\mathcal{j}=\langle\lambda,h\rangle$.
Conversely, for any restricted $\tilde{\mathfrak{g}}$-module $M$ of level $\mathcal{l}\neq-\frac{3}{2}$,
by Sugawara construction we can extend $M$ to a $\hat{\mathfrak{g}}$-module by letting $d$ act on $M$ as $-L(0)$.
In this paper we shall consider any restricted $\tilde{\mathfrak{g}}$-module as a $\hat{\mathfrak{g}}$-module in this way.

Let $\mathcal{l}\in\mathbb{C}$ and $U$ be a $\mathbb{C}h$-module,
$U$ can be regarded as a $B$-module by setting $N_+$ acting trivially and $k$ acting as $\mathcal{l}$,
let $M(\mathcal{l},U)=U(\tilde{\mathfrak{g}})\otimes_{U(B)}U$.
If $U=\mathbb{C}$ is a one-dimensional $\mathbb{C}h$-module on which $h$ acts as a fixed complex number $\mathcal{j}$,
the corresponding module is a Verma module denoted by $M(\mathcal{l},\mathcal{j})$.
Then $M(\mathcal{l},\mathcal{j})$ has a unique maximal submodule and $L(\mathcal{l},\mathcal{j})$ is isomorphic to the corresponding irreducible quotient.
Similarly we can define the generalized Verma $\tilde{\mathfrak{g}}$-module $V(\mathcal{l},U)=U(\tilde{\mathfrak{g}})\otimes_{U(P)}U$
for any $\mathfrak{g}$-module $U$ which can be extend to a $P$-module by setting $\mathfrak{g}\otimes t\mathbb{C}[t]$
acting trivially and $k$ acting as $\mathcal{l}$.
Note that if $U=\mathbb{C}$ is the trivial $\mathfrak{g}$-module,
then $V(\mathcal{l},\mathbb{C})$ is a quotient of $M(\mathcal{l},0)$
and $L(\mathcal{l},0)$ is isomorphic to the irreducible quotient of $V(\mathcal{l},\mathbb{C})$.

We recall the following reducibility criterion in \cite{IK} which generalized the Kac-Kazhdan reducibility criterion \cite{KK}.
\begin{prop}
{\rm\cite{IK}}. The Verma module $M(\mathcal{l},\mathcal{j})(\mathcal{l}\neq-\frac{3}{2})$ is reducible if and only if
$$\mathcal{j}=\frac{m-1}{2}-s(\mathcal{l}+\frac{3}{2}),$$
where $m,s\in\mathbb{Z}$ such that $m+s\in2\mathbb{Z}+1$ and either $m>0,s\geq0$ or $m<0,s<0$.
\end{prop}

In \cite{KW}, Kac and Wakimoto gave the definition of admissible weight for Kac-Moody (super)algebras
and charactered the admissible weight of $\hat{\mathfrak{g}}$ as an example.
There is an equivalent characterization of the admissible weight of $\hat{\mathfrak{g}}$ in \cite{IK}.

\begin{prop}\label{propadm}
{\rm\cite{IK}}. For any $\lambda\in H^*$, let $\mathcal{l}=\langle\lambda,k\rangle,\mathcal{j}=\langle\lambda,h\rangle$.\\
{\rm (1)} $\mathcal{l}$ is an admissible level if and only if $\mathcal{l}+\frac{3}{2}=\frac{p}{2q}$,
where $p,q$ are positive integers such that $p\geq2, p\equiv q(\mbox{mod}~2)$ and $(\frac{p-q}{2},q)=1$.\\
{\rm (2)} For an admissible level $\mathcal{l}$, $\mathcal{j}$ is an admissible weight if and only if
$\mathcal{j}=\frac{m-1}{2}-ls$, where $l=\mathcal{l}+\frac{3}{2}=\frac{p}{2q},m,s\in \mathbb{Z}$ such that $m+s\equiv1(\mbox{mod}~2)$ and $1\leq m\leq p-1, 0\leq s \leq q-1$.
\end{prop}

From now on we will assume that $l=\mathcal{l}+\frac{3}{2}=\frac{p}{2q}$, where $p,q$ are positive integers such that $p\geq2, p\equiv q(\mbox{mod}~2)$ and $(\frac{p-q}{2},q)=1$.

\begin{rem}\label{remj}
{\rm Let $\mathcal{j}=\frac{m-1}{2}-ls=\frac{m^\prime-1}{2}-ls^\prime$ be an admissible weight such that
\begin{equation}\nonumber
\begin{aligned}
&m,m^\prime,s,s^\prime\in \mathbb{Z},1\leq m\leq p-1, 1\leq m^\prime\leq p-1,0\leq s^\prime\leq s \leq q-1,\\
&m+s\equiv1(\mbox{mod}~2),m^\prime+s^\prime\equiv1(\mbox{mod}~2).
\end{aligned}
\end{equation}
Suppose $s>s^\prime$, then $\frac{m-m^\prime}{s-s^\prime}=\frac{p}{q}$.
Since $p\equiv q(\mbox{mod}~2)$, we have $p,q\in2\mathbb{Z}$ or $p,q\in2\mathbb{Z}+1$.
If $p,q\in2\mathbb{Z}+1$, by $(\frac{p-q}{2},q)=1$ we have $(p,q)=1$, it is a contradiction.
If $p,q\in2\mathbb{Z}$, by $(\frac{p-q}{2},q)=1$ we have $s-s^\prime=\frac{1}{2}q, m-m^\prime=\frac{1}{2}p$,
then $(m-m^\prime, s-s^\prime)=1$ and $m-m^\prime+s-s^\prime\equiv1(\mbox{mod}~2)$,
it contradicts to $m+s\equiv1(\mbox{mod}~2)$ and $m^\prime+s^\prime\equiv1(\mbox{mod}~2)$.
Hence $s=s^\prime, m=m^\prime$, it shows that the expression $\mathcal{j}=\frac{m-1}{2}-ls$ of an admissible weight $\mathcal{j}$ with
$m,s\in \mathbb{Z}, m+s\equiv1(\mbox{mod}~2)$ and $1\leq m\leq p-1, 0\leq s \leq q-1$ is unique.}
\end{rem}

A vector $w$ in a highest weight module $M$ for $\hat{\mathfrak{g}}$ is called a {\em singular vector} if
$w$ as a highest weight vector generates a proper submodule of $M$.
Note that $l=\frac{p}{2q}$ and $\mathcal{j}=\frac{m-1}{2}-ls=\frac{m-p-1}{2}-l(s-q)$, then from Remark \ref{remj},
together with the proof of Theorem 4.2 in \cite{IK} and Corollary 1 in \cite{KW},
we have the following proposition.

\begin{prop}\label{propsv}
Let $\mathcal{j}=\frac{m-1}{2}-ls$, where $m,s\in \mathbb{Z}$ with $m+s\equiv1(\mbox{mod}~2)$ and $1\leq m\leq p-1, 0\leq s \leq q-1$
and let $v$ be a highest weight vector which generates the Verma module $M(\mathcal{l},\mathcal{j})$. Set
\begin{equation}
\begin{aligned}
&\begin{aligned}F_1(m,s)=&y(0)(y(0)^2)^{\frac{m-1}{2}+ls}e(-1)^{\frac{m}{2}+l(s-1)}y(0)(y(0)^2)^{\frac{m-1}{2}+l(s-2)}\cdots\\
 &e(-1)^{\frac{m}{2}-l(s-1)}y(0)(y(0)^2)^{\frac{m-1}{2}-ls},
 \end{aligned}\\
&\begin{aligned}
F_2(m,s)=&e(-1)^{\frac{p-m}{2}+l(q-s-1)}y(0)(y(0)^2)^{\frac{p-m-1}{2}+l(q-s-2)}e(-1)^{\frac{p-m}{2}+l(q-s-3)}\cdots\\ &y(0)(y(0)^2)^{\frac{p-m-1}{2}-l(q-s-2)}e(-1)^{\frac{p-m}{2}-l(q-s-1)}.
\end{aligned}
\end{aligned}
\end{equation}
Then $v_{j,1}=F_1(m,s)v, v_{j,2}=F_2(m,s)v$ are singular vectors of $M(\mathcal{l},\mathcal{j})$.
Moreover, the maximal proper submodule of $M(\mathcal{l},\mathcal{j})$ is generated by $v_{\mathcal{j},1}$ and $v_{\mathcal{j},2}$.
\end{prop}

We regard $(y(0)^2)^{\gamma}$ as an even element for any $\gamma\in\mathbb{C}$.
It follows from \cite{MFF} that $F_1(m,s)$ and $F_2(m,s)$ make sense as an element of $U(\tilde{\mathfrak{g}})$, and we have
\begin{equation}
  a^\gamma b=ba^\gamma+\sum_{j>0}\begin{pmatrix}\gamma\\j\end{pmatrix}(\mbox{ad}a)^j(b)a^{\gamma-j}
\end{equation}
for $a\in\tilde{\mathfrak{g}}_{\bar{0}},b\in U(\tilde{\mathfrak{g}}),\gamma\in\mathbb{C}$.

\begin{rem}\label{remj=0}
{\rm For $\mathcal{j}=0$, from Remark \ref{remj}, we have $m=1,s=0$.
By Proposition \ref{propsv}, $v_{0,1}=F_1(1,0)\textbf{1}, v_{0,2}=F_2(1,0)\textbf{1}$ are singular vectors of $M(\mathcal{l},0)$.
Since $v_{0,1}=F_1(1,0)\textbf{1}=0$ in $V(\mathcal{l},\mathbb{C})$, we have $v_{0,2}=F_2(1,0)\textbf{1}$ generates the maximal proper submodule of $V(\mathcal{l},\mathbb{C})$.}
\end{rem}

Set $P(\alpha)=xy+\alpha$ and $Q(\alpha)=yx-\alpha$ for $\alpha\in\mathbb{C}$.
We have the following commutation relations in $U(\mathfrak{g})$ (cf. \cite{IK}).

\begin{prop}\label{proppai1}
For any $\alpha,\beta,\gamma\in\mathbb{C}$, we have
\begin{equation}
\begin{aligned}
&[P(\alpha),P(\beta)]=0,[P(\alpha),Q(\beta)]=0,[Q(\alpha),Q(\beta)]=0,\\
&e^\gamma P(\alpha)=P(\alpha-\gamma)e^\gamma, e^\gamma Q(\alpha)=Q(\alpha+\gamma)e^\gamma,\\
&f^\gamma P(\alpha)=P(\alpha+\gamma)f^\gamma, f^\gamma Q(\alpha)=P(\alpha-\gamma)f^\gamma,\\
&xP(\alpha)=Q(1-\alpha)x, xQ(\alpha)=P(-\alpha)x,
yP(\alpha)=Q(-\alpha)y, yQ(\alpha)=P(1-\alpha)y,\\
&e^\gamma y=Q(\gamma)xe^{\gamma-1},xe^\gamma y=P(-\gamma)e^\gamma,
f^\gamma x=-P(\gamma)yf^{\gamma-1},yf^\gamma x=Q(-\gamma)f^\gamma.
\end{aligned}
\end{equation}
\end{prop}

Let $\sigma$ be the anti-automorphism of superalgebra $U(\mathfrak{g})$ such that $\sigma(a)=-a$ for any $a\in\mathfrak{g}$.
Then $\sigma(P(\alpha))=-Q(\alpha)$ and $\sigma(Q(\alpha))=-P(\alpha)$ for any $\alpha\in\mathbb{C}$.
Let $\pi$ be the projection $\tilde{\mathfrak{g}}$ onto $\mathfrak{g}$ such that $\pi(a\otimes t^n)=a$ for any $a\in\mathfrak{g}$ and $\pi(c)=0$.
From the Lemma 4.2 in \cite{IK}, we have the following proposition.

\begin{prop}\label{proppai}
{\rm\cite{IK}}.The following projection formulas hold:
\begin{equation}
\begin{aligned}
&\pi(F_1(m,s))=\prod_{j=1}^s\prod_{i=1}^m\{\prod_{i+j\in2\mathbb{Z}}P(\frac{i}{2}+jl)\prod_{i+j\in2\mathbb{Z}+1}Q(-\frac{i-1}{2}-jl)\}y^m,\\
&\pi(F_2(m,s))=\prod_{j=1}^{q-s-1}\prod_{i=1}^{p-m}\{\prod_{i+j\in2\mathbb{Z}}Q(\frac{i}{2}+jl)\prod_{i+j\in2\mathbb{Z}+1}P(-\frac{i-1}{2}-jl)\}x^{p-m}.
\end{aligned}
\end{equation}
\end{prop}

\section{Rationality of $(L_{\widehat{\mathfrak{g}}}(\mathcal{l},0),\omega_\xi)$}
\label{sec:3}
	\def\theequation{3.\arabic{equation}}
	\setcounter{equation}{0}

In this section,
we first construct the weak $V_{\widehat{\mathfrak{g}}}(\mathcal{l},\mathbb{C})$-module $L(\mathcal{l},U)$
and determine the condition that $L(\mathcal{l},U)$ is a weak $L_{\widehat{\mathfrak{g}}}(\mathcal{l},0)$-module.
Then we prove that $L_{\widehat{\mathfrak{g}}}(\mathcal{l},0)$ at admissible level is rational
in category $\mathcal{C}_{\mathcal{l}}$ and $(L_{\widehat{\mathfrak{g}}}(\mathcal{l},0),\omega_\xi)$
are rational $\mathbb{Q}$-graded vertex operator superalgebras respect to a family of Virasoro elements $\omega_\xi$, where $0<\xi<1$ is a rational number.

\subsection{The weak $L_{\widehat{\mathfrak{g}}}(\mathcal{l},0)$-module $L(\mathcal{l},U)$}
\label{sec:3.1}

We know that for $\mathcal{l}\neq -\frac{3}{2}$,
$V(\mathcal{l},\mathbb{C})$ and $L(\mathcal{l},0)$ have natural
$\mathbb{Z}$-graded vertex operator superalgebra structures
and any $M(\mathcal{l},U)$ is a weak module for $V(\mathcal{l},\mathbb{C})$.
We denote by $L_{\widehat{\mathfrak{g}}}(\mathcal{l},0)$ the $\mathbb{Z}$-graded vertex operator superalgebra $L(\mathcal{l},0)$.
For a $\mathbb{C}h$-module $U$, we define a linear function on $U^*\otimes M(\mathcal{l},U)$ as follows:
$$\langle u^\prime, u\rangle=u^\prime(\pi^\prime(u)), ~\mbox{for}~~u^\prime\in U^*,u\in M(\mathcal{l},U),$$
where $\pi^\prime$ is the projection of $M(\mathcal{l},U)$ onto the subspace $U$.
Define
\begin{equation}
I=\{u\in M(\mathcal{l},U)\mid\langle u^\prime, au\rangle=0~~\mbox{for~~any}~~u^\prime\in U^*,a\in U(\tilde{\mathfrak{g}})\}.
\end{equation}

\begin{lem}\label{lemI}
$I$ is unique and maximal in the set of submodules of $M(\mathcal{l},U)$ which intersects with $U$ trivially.
\end{lem}
\begin{proof}
Let $I^\prime$ is a submodule of $M(\mathcal{l},U)$ which intersects with $U$ trivially.
Suppose $I^\prime\nsubseteq I$,
we have there exists $u\in I^\prime$ such that
$\langle u^\prime, au\rangle\ne0$ for some $u^\prime\in U^*$, then $\pi^\prime(au)\ne0$.
Since $I^\prime$ is a submodule, we have $au\in I^\prime$.
From the natural gradation of $M(\mathcal{l},U)$ as a $\hat{\mathfrak{g}}$-module,
we have $\pi^\prime(au)\in I^\prime$, it contradicts to that $I^\prime$ intersects with $U$ trivially, hence $I^\prime\subseteq I$.
Therefore $I$ is unique and maximal in the set of submodules of $M(\mathcal{l},U)$ which intersects with $U$ trivially.
\end{proof}

Set $L(\mathcal{l},U)=M(\mathcal{l},U)/I$ and we regard $U$ as a subspace of $L(\mathcal{l},U)$.
Then $\pi^\prime$ induces a projection of $L(\mathcal{l},U)$ to $U$, which we still denote it by $\pi^\prime$.
It is clear that $L(\mathcal{l},U)$ is a weak module for $V(\mathcal{l},\mathbb{C})$.
Let $Y(\cdot,z)$ be the vertex operator map defining the module structure on $L(\mathcal{l},U)$,
then $Y(\cdot,z)$ is an intertwining operator of type $\begin{pmatrix}L(\mathcal{l},U)\\ V(\mathcal{l},\mathbb{C})L(\mathcal{l},U)\end{pmatrix}$.
Let
$$\mathcal{Y}(u,z)v=(-1)^{|u||v|}e^{zL(-1)}Y(v,-z)u$$
for any homogeneous element $u\in L(\mathcal{l},U),v\in V(\mathcal{l},\mathbb{C})$,
then $\mathcal{Y}(\cdot,z)$ is an intertwining operator of type
$\begin{pmatrix}L(\mathcal{l},U)\\ L(\mathcal{l},U)V(\mathcal{l},\mathbb{C})\end{pmatrix}$ (cf. \cite{FHL}).

\begin{lem}\label{lem111}
The $\tilde{\mathfrak{g}}$-module $L(\mathcal{l},U)$ is a weak module for $L_{\widehat{\mathfrak{g}}}(\mathcal{l},0)$
if and only if
\begin{equation}
\langle u^\prime,\mathcal{Y}(u,z)v_{0,2}\rangle=0~~\mbox{for~~any}~~u^\prime\in U^*, u\in U(\mathfrak{g})U\subset L(\mathcal{l},U).
\end{equation}
\end{lem}
\begin{proof}
Let $J$ be the maximal submodule of $V(\mathcal{l},\mathbb{C})$. From Remark \ref{remj=0}, we have $J=U(N_-)v_{0,2}$.
Since $L_{\widehat{\mathfrak{g}}}(\mathcal{l},0)=V(\mathcal{l},\mathbb{C})/J$ is the quotient vertex operator superalgebra of $V(\mathcal{l},\mathbb{C})$
and $L(\mathcal{l},U)$ is a weak module for the vertex operator superalgebra $L_{\widehat{\mathfrak{g}}}(\mathcal{l},0)$,
then for any $v\in J$ we have $Y(v,z)L(\mathcal{l},U)=0$.
Hence $$\mathcal{Y}(u,z)v_{0,2}=(-1)^{|u||v_{0,2}|}e^{zL(-1)}Y(v_{0,2},-z)u=0$$ for any homogeneous element $u\in L(\mathcal{l},U)$,
then $\langle u^\prime,\mathcal{Y}(u,z)v_{0,2}\rangle=0~~\mbox{for~~any}~~u^\prime\in U^*, u\in U(\mathfrak{g})U\subset L(\mathcal{l},U)$.

For the other hand, if $a\in N_-U(N_-)$, we have $\pi^\prime(a\mathcal{Y}(u,z)w)=0$. Then
$$\langle u^\prime,a\mathcal{Y}(u,z)w\rangle=0~~\mbox{for~~any}~~u^\prime\in U^*, u\in L(\mathcal{l},U), a\in N_-U(N_-), w\in V(\mathcal{l},\mathbb{C}).$$
Let $\Gamma=\{w\in V(\mathcal{l},\mathbb{C})\mid \langle u^\prime,\mathcal{Y}(u,z)w\rangle=0~~\mbox{for~~any}~~u^\prime\in U^*, u\in U(\mathfrak{g})U\subset L(\mathcal{l},U)\}$, then $v_{0,2}\in \Gamma$.
From the Jacobi identity for the intertwining operator we can get the following commutator formula
\begin{equation}\label{cf}
[a(m),\mathcal{Y}(u,z)]=\sum_{j\geq0}\begin{pmatrix}m\\j\end{pmatrix}\mathcal{Y}(a(j)u,z)z^{m-j}
\end{equation}
for any $a\in\mathfrak{g}, m\in\mathbb{Z}$ and $u\in L(\mathcal{l},U)$.
If $w\in\Gamma, a(m)\in N_-, u\in U(\mathfrak{g})U\subset L(\mathcal{l},U)$,
\begin{equation*}
\begin{aligned}
\mathcal{Y}(u,z)a(m)w&=(-1)^{|u||a|}(a(m)\mathcal{Y}(u,z)w-[a(m),\mathcal{Y}(u,z)]w)\\
&=(-1)^{|u||a|}(a(m)\mathcal{Y}(u,z)w-\sum_{j\geq0}\begin{pmatrix}m\\j\end{pmatrix}\mathcal{Y}(a(j)u,z)z^{m-j}w),
\end{aligned}
\end{equation*}
then by $j\geq0$ we have $a(j)u\in U(\mathfrak{g})U\subset L(\mathcal{l},U)$ and $\langle u^\prime,\mathcal{Y}(u,z)a(m)w\rangle=0$,
hence $a(m)w\in\Gamma$.
Then we have $J\subseteq\Gamma$.
From the Jacobi identity we can also get
\begin{equation}
\mathcal{Y}(a(n)u,z)=\sum_{j\geq0}\begin{pmatrix}n\\j\end{pmatrix}a(n-j)\mathcal{Y}(u,z)z^j-(-1)^{n+|u||a|}\sum_{j\geq0}\begin{pmatrix}n\\j\end{pmatrix}\mathcal{Y}(u,z)a(j)z^{n-j}
\end{equation}
for any $u\in L(\mathcal{l},U),a\in\mathfrak{g}, n\in\mathbb{Z}$.
Since $L(\mathcal{l},U)$ is generated by $U$ as $\tilde{\mathfrak{g}}$-module,
then we have $\langle u^\prime,\mathcal{Y}(u,z)J\rangle=0~~\mbox{for~~any}~~u^\prime\in U^*, u\in L(\mathcal{l},U)$.
From (\ref{cf}), we have
$$\langle u^\prime,x\mathcal{Y}(u,z)J\rangle=0~~\mbox{for~~any}~~u^\prime\in U^*, x\in U(\tilde{\mathfrak{g}}), u\in L(\mathcal{l},U).$$
Then $\mathcal{Y}(u,z)J\subseteq I$ for any $u\in L(\mathcal{l},U)$, i.e., $\mathcal{Y}(u,z)J=0$ in $L(\mathcal{l},U)$.
 It implies that $\mathcal{Y}(\cdot,z)$ induces an intertwining operator of type $\begin{pmatrix}L(\mathcal{l},U)\\ L(\mathcal{l},U)L_{\widehat{\mathfrak{g}}}(\mathcal{l},0)\end{pmatrix}$.
Hence $L(\mathcal{l},U)$ is a weak module for the vertex operator superalgebra $L_{\widehat{\mathfrak{g}}}(\mathcal{l},0)$.
\end{proof}

\begin{thm}\label{thm1}
$L(\mathcal{l},U)$ is a weak $L_{\widehat{\mathfrak{g}}}(\mathcal{l},0)$-module if and only if $f(h)U=0$, where
\begin{equation}\label{fh}
f(h)=\prod_{i+j\in2\mathbb{Z}+1,0\leq j\leq q-1,1\leq i\leq p-1}(h-\frac{i-1}{2}+jl).
\end{equation}
\end{thm}
\begin{proof}
For $n\in\mathbb{Z}, a\in\mathfrak{g}$, we define $\mbox{deg}(a\otimes t^n)=n$.
For any $u^\prime\in U^*, a(m)\in N_-, u\in U(\mathfrak{g})U\subset L(\mathcal{l},U), w\in V(\mathcal{l},\mathbb{C})$, by the commutator formula (\ref{cf}) we have
\begin{align*}
 \langle u^\prime,\mathcal{Y}(u,z)a(m)w\rangle&=(-1)^{|u||a|}(\langle u^\prime,a(m)\mathcal{Y}(u,z)w\rangle-\langle u^\prime,[a(m),\mathcal{Y}(u,z)]w\rangle)\\
 &=(-1)^{|u||a|+1}\langle u^\prime,\sum_{j\geq0}\begin{pmatrix}m\\j\end{pmatrix}\mathcal{Y}(a(j)u,z)z^{m-j}w\rangle\\
 &=(-1)^{|u||a|+1}\langle u^\prime,\mathcal{Y}(a(0)u,z)z^{m}w\rangle\\
 &=\langle u^\prime,(-1)^{|u||a|+1}z^{\mbox{deg}(a(m))}\mathcal{Y}(\pi(a(m))u,z)w\rangle.
\end{align*}
Hence for any $u^\prime\in U^*, a\in U(N_-), u\in U(\mathfrak{g})U\subset L(\mathcal{l},U), w\in V(\mathcal{l},\mathbb{C})$ we have
$$ \langle u^\prime,\mathcal{Y}(u,z)aw\rangle=\langle u^\prime,\xi z^{\mbox{deg}(a)}\mathcal{Y}(\sigma\pi(a)u,z)w\rangle,$$
where $\xi=1$ or $-1$.
Let $a=F_2(1,0)$, then $v_{0,2}=a\textbf{1}$.
From Lemma \ref{lem111}, $L(\mathcal{l},U)$ is a weak $L_{\widehat{\mathfrak{g}}}(\mathcal{l},0)$-module if and only if
\begin{equation*}
\langle u^\prime,\mathcal{Y}(\sigma\pi(a)u,z)\textbf{1}\rangle=0~~\mbox{for~~any}~~u^\prime\in U^*, u\in U(\mathfrak{g})U\subset L(\mathcal{l},U).
\end{equation*}
By Proposition \ref{proppai1} and Proposition \ref{proppai} we have
\begin{align*}
&\sigma\pi(a)=\sigma\prod_{j=1}^{q-1}\prod_{i=1}^{p-1}\{\prod_{i+j\in2\mathbb{Z}}Q(\frac{i}{2}+jl)\prod_{i+j\in2\mathbb{Z}+1}P(-\frac{i-1}{2}-jl)\}x^{p-1}\\
&=(-1)^{q(p-1)}x^{p-1}\prod_{j=1}^{q-1}\prod_{i=1}^{p-1}\{\prod_{i+j\in2\mathbb{Z}}P(\frac{i}{2}+jl)\prod_{i+j\in2\mathbb{Z}+1}Q(-\frac{i-1}{2}-jl)\}\\
&=\begin{cases}
(-1)^{q(p-1)}\prod_{j=1}^{q-1}\prod_{i=1}^{p-1}\{\prod_{i+j\in2\mathbb{Z}}P(-\frac{p-i-1}{2}+jl)\prod_{i+j\in2\mathbb{Z}+1}Q(\frac{p-i}{2}-jl)\}x^{p-1}, &p\notin2\mathbb{Z}\\
(-1)^{q(p-1)}\prod_{j=1}^{q-1}\prod_{i=1}^{p-1}\{\prod_{i+j\in2\mathbb{Z}}Q(\frac{p-i}{2}-jl)\prod_{i+j\in2\mathbb{Z}+1}P(-\frac{p-i-1}{2}+jl)\}x^{p-1},
&p\in2\mathbb{Z}
\end{cases}\\
&=(-1)^{q(p-1)}\prod_{j=1}^{q-1}\prod_{i=1}^{p-1}\{\prod_{i+j\in2\mathbb{Z}+1}P(-\frac{i-1}{2}+jl)\prod_{i+j\in2\mathbb{Z}}Q(\frac{i}{2}-jl)\}x^{p-1}.
\end{align*}
Note that $\mathcal{Y}(\sigma\pi(a)u,z)\textbf{1}=e^{zL(-1)}Y(\textbf{1},-z)\sigma\pi(a)u=e^{zL(-1)}\sigma\pi(a)u$.
Then for any $u^\prime\in U^*, u\in U(\mathfrak{g})U\subset L(\mathcal{l},U)$,
\begin{equation*}
\langle u^\prime,\mathcal{Y}(\sigma\pi(a)u,z)\textbf{1}\rangle=\langle u^\prime,e^{zL(-1)}\sigma\pi(a)u\rangle=\langle u^\prime,\sigma\pi(a)u\rangle,
\end{equation*}
thus $L(\mathcal{l},U)$ is a weak $L_{\widehat{\mathfrak{g}}}(\mathcal{l},0)$-module if and only if for any $u^\prime\in U^*$,
\begin{equation*}
\langle u^\prime,
\prod_{j=1}^{q-1}\prod_{i=1}^{p-1}\{\prod_{i+j\in2\mathbb{Z}+1}P(-\frac{i-1}{2}+jl)\prod_{i+j\in2\mathbb{Z}}Q(\frac{i}{2}-jl)\}x^{p-1}U(\mathfrak{g})U\rangle=0.
\end{equation*}
From the grading restriction on the bilinear pair, it is equivalent to
\begin{equation*}
\prod_{j=1}^{q-1}\prod_{i=1}^{p-1}\{\prod_{i+j\in2\mathbb{Z}+1}P(-\frac{i-1}{2}+jl)\prod_{i+j\in2\mathbb{Z}}Q(\frac{i}{2}-jl)\}x^{p-1}y^{p-1}U=0.
\end{equation*}
By Proposition \ref{proppai1}, we have
\begin{align*}
x^{p-1}y^{p-1}
&=\begin{cases}
\prod_{i=1}^{\frac{p-1}{2}}Q(\frac{p+1}{2}-i)P(-\frac{p-1}{2}+i), &p\notin2\mathbb{Z}\\
P(0)\prod_{i=1}^{\frac{p-2}{2}}P(-\frac{p}{2}+i)Q(\frac{p}{2}-i),  &p\in2\mathbb{Z}
\end{cases}\\
&=\prod_{i\in2\mathbb{Z}+1,1\leq i\leq p-1}P(-\frac{i-1}{2})\prod_{i\in2\mathbb{Z},1\leq i\leq p-1}Q(\frac{i}{2}).
\end{align*}
Since $xU=0$, we have for any $\alpha\in\mathbb{C}$,
$$P(\alpha)U=xyU+\alpha U=(h+\alpha)U, Q(\alpha)U=-\alpha U.$$
Then $L(\mathcal{l},U)$ is a weak $L_{\widehat{\mathfrak{g}}}(\mathcal{l},0)$-module if and only if
\begin{equation*}
\prod_{i+j\in2\mathbb{Z}+1,1\leq j\leq q-1,1\leq i\leq p-1}(h-\frac{i-1}{2}+jl)\prod_{i\in2\mathbb{Z}+1,1\leq i\leq p-1}(h-\frac{i-1}{2})U=0.
\end{equation*}
It is equivalent to
 $$f(h)U=\prod_{i+j\in2\mathbb{Z}+1,0\leq j\leq q-1,1\leq i\leq p-1}(h-\frac{i-1}{2}+jl)U=0.$$
\end{proof}

From Proposition \ref{propadm} and Theorem \ref{thm1}, we obtain the following corollary.

\begin{cor}\label{coro1}
The $\tilde{\mathfrak{g}}$-module $L(\mathcal{l},\mathcal{j})$ is a weak $L_{\widehat{\mathfrak{g}}}(\mathcal{l},0)$-module if and only if $\mathcal{j}$ is admissible.
\end{cor}

Let $\mathcal{O}_{\mathcal{l}}$ be the subcategory of the weak $L_{\widehat{\mathfrak{g}}}(\mathcal{l},0)$-module category
such that $M$ is an object in $\mathcal{O}_{\mathcal{l}}$ if and only if $M\in \mathcal{O}$ as $\hat{\mathfrak{g}}$-module.
In \cite{W}, Wood showed that $\mathcal{O}_{\mathcal{l}}$ is semisimple.

\begin{thm}\label{thmo}
{\rm\cite{W}}.
 The weak $L_{\widehat{\mathfrak{g}}}(\mathcal{l},0)$-modules in the category $\mathcal{O}$ are completely reducible,
 i.e., $\mathcal{O}_{\mathcal{l}}$ is semisimple. Moreover,
 the irreducible modules correspond to $L(\mathcal{l},\mathcal{j})$,
 where $\mathcal{j}$ is admissible.
\end{thm}

Corollary \ref{coro1} can also be obtained by Theorem \ref{thmo}.
Conversely, by Corollary \ref{coro1} and the proof of Theorem 4.1 in \cite{KW2},
we can obtain Theorem \ref{thmo} in a different way.
From Theorem \ref{thmo},
for any $M\in\mathcal{O}_{\mathcal{l}}$,
let $U=\{w\in M\mid N_+w=0\}$,
it is clear that $U$ is a $\mathbb{C}h$-module
and $M\cong L(\mathcal{l},U)$.

\subsection{Category $\mathcal{C}_{\mathcal{l}}$}
\label{sec:3.2}

\begin{defi}
{\em The category $\mathcal{C}_{\mathcal{l}}$ is the subcategory of the weak $L_{\widehat{\mathfrak{g}}}(\mathcal{l},0)$-module category
such that $M$ is an object in $\mathcal{C}_{\mathcal{l}}$ if and only if $a$ acts locally nilpotently on $M$ for all $a\in N_+$.}
\end{defi}

It is clear that the object in $\mathcal{O}_{\mathcal{l}}$ is also an object in $\mathcal{C}_{\mathcal{l}}$.
We will show that $\mathcal{C}_{\mathcal{l}}$ is semisimple. Let
$\Omega(M)=\{v\in M\mid (\mathfrak{g}\otimes t\mathbb{C}[t]).v=0\}$
for a weak $L_{\widehat{\mathfrak{g}}}(\mathcal{l},0)$-module $M$.
We can prove the following proposition.

\begin{prop}\label{prop11}
Let $M$ be a weak $L_{\widehat{\mathfrak{g}}}(\mathcal{l},0)$-module belonging to $\mathcal{C}_{\mathcal{l}}$, then $M$ has a highest weight vector.
\end{prop}
\begin{proof}
Similar to the Proposition 3.6 of \cite{Lin} or the Theorem 3.7 of \cite{DLM1}, we have $\Omega(M)\ne0$.
Since $M$ is an object of $\mathcal{C}_{\mathcal{l}}$,
we have $e$ and $x$ act locally nilpotent on $\Omega(M)$.
Let $\mathfrak{n}_+=\mathbb{C}e\oplus\mathbb{C}x$, then $U(\mathfrak{n}_+)v$ is finite dimensional for any $v\in\Omega(M)$.
Since $\mathfrak{n}_+$ is a nilpotent Lie superalgebra,
then by the Engel theorem for Lie superalgebra (cf. \cite{K}), there exists a nonzero $w\in U(\mathfrak{n}_+)v$ such that $\mathfrak{n}_+.w=0$.
Since $\mathfrak{n}_+.\Omega(M)\subseteq\Omega(M)$, we have $w\in\Omega(M)$, then $N_+.w=0$.
Furthermore,
let $U=U(\mathfrak{g})w$ and $W$ be the weak $L_{\widehat{\mathfrak{g}}}(\mathcal{l},0)$-submodule generated by $U$.
Then $W$ is isomorphic to a quotient module of $M(\mathcal{l},U)$ as $\tilde{\mathfrak{g}}$-module.
From Lemma \ref{lemI}, $L(\mathcal{l},U)$ is isomorphic to some quotient module of $W$ as $\tilde{\mathfrak{g}}$-module,
then $L(\mathcal{l},U)$ is a quotient of $W$ as weak $V(\mathcal{l},\mathbb{C})$-module.
Therefore $L(\mathcal{l},U)$ is a weak $L_{\widehat{\mathfrak{g}}}(\mathcal{l},0)$-module.
From Theorem \ref{thm1}, we have $f(h)U=0$,
then $h$ semisimply acts on $U$, i.e., $w$ is a highest weight vector.
\end{proof}

We now classify simple objects in the category $\mathcal{C}_{\mathcal{l}}$.

\begin{prop}\label{propirr}
Any irreducible weak $L_{\widehat{\mathfrak{g}}}(\mathcal{l},0)$-module belonging to $\mathcal{C}_{\mathcal{l}}$ is of the form
$L(\mathcal{l},\mathcal{j})$ with admissible weight $\mathcal{j}$.
Moreover, up to isomorphism, there are finitely many irreducible weak $L_{\widehat{\mathfrak{g}}}(\mathcal{l},0)$-modules belonging to $\mathcal{C}_{\mathcal{l}}$.
\end{prop}
\begin{proof}
Let $M$ be an irreducible weak $L_{\widehat{\mathfrak{g}}}(\mathcal{l},0)$-module belonging to $\mathcal{C}_{\mathcal{l}}$,
then by Proposition \ref{prop11}, $M$ contains a highest weight vector $w$.
Since $M$ is an irreducible weak $L_{\widehat{\mathfrak{g}}}(\mathcal{l},0)$-module,
we have $M$ is an irreducible weak $V(\mathcal{l},\mathbb{C})$-module,
then $M$ is an irreducible $\tilde{\mathfrak{g}}$-module generated by $w$.
Therefore $M$ is an irreducible highest weight module of $\tilde{\mathfrak{g}}$. By Corollary \ref{coro1},
$M$ is of the form
$L(\mathcal{l},\mathcal{j})$ for some admissible weight $\mathcal{j}$.
\end{proof}

Then we can prove the category $\mathcal{C}_{\mathcal{l}}$ is semisimple.

\begin{thm}\label{thmcc}
The category $\mathcal{C}_{\mathcal{l}}$ is semisimple.
Moreover,
any object $M$ in $\mathcal{C}_{\mathcal{l}}$ is a direct sum of $L(\mathcal{l},\mathcal{j})$ with admissible weight $\mathcal{j}$.
\end{thm}
\begin{proof}
Let $M$ be a weak $L_{\widehat{\mathfrak{g}}}(\mathcal{l},0)$-module belonging to $\mathcal{C}_{\mathcal{l}}$.
From Proposition \ref{prop11}, $M$ contains a highest weight vector $w$.
Let $W$ be the weak $L_{\widehat{\mathfrak{g}}}(\mathcal{l},0)$-submodule generated by $w$,
then $W$ is a quotient of certain Verma module $M(\mathcal{l},\mathcal{j}^\prime)$ as $\tilde{\mathfrak{g}}$-module.
Hence $W$ is an object in $\mathcal{O}_{\mathcal{l}}$.
From Theorem \ref{thmo}, it is completely reducible and $W=\bigoplus_{\mathcal{j}} L(\mathcal{l},\mathcal{j})$ for admissible weights $\mathcal{j}$.
But the highest weight subspace of $M(\mathcal{l},\mathcal{j}^\prime)$ is one dimensional,
then $W=L(\mathcal{l},\mathcal{j}^\prime)$ with admissible weight $\mathcal{j}^\prime$.
Let $W^\prime$ be the sum of irreducible weak $L_{\widehat{\mathfrak{g}}}(\mathcal{l},0)$-module of $M$,
then $W^\prime$ is a direct sum of $L(\mathcal{l},\mathcal{j})$ for admissible weights $\mathcal{j}$.

If $W^\prime$ is a proper submodule of $M$,
then $M/W^\prime$ is a weak $L_{\widehat{\mathfrak{g}}}(\mathcal{l},0)$-module belonging to $\mathcal{C}_{\mathcal{l}}$,
thus it contains a highest weight vector $\bar{w^\prime}$ by Proposition \ref{prop11}.
Let $w^\prime$ be a preimage of $\bar{w^\prime}$, then $N_+.w^\prime\subseteq W^\prime$.
Since $N_+$ is finitely generated,
there exist submodules $L(\mathcal{l},\mathcal{j}_1), L(\mathcal{l},\mathcal{j}_2),\cdots, L(\mathcal{l},\mathcal{j}_s)$ of $W^\prime$
such that $$N_+.w^\prime\subseteq L(\mathcal{l},\mathcal{j}_1)\oplus L(\mathcal{l},\mathcal{j}_2)\oplus\cdots\oplus L(\mathcal{l},\mathcal{j}_s).$$
From Theorem \ref{thmo},
the submodule of $M$ generated by $w^\prime$ and $L(\mathcal{l},\mathcal{j}_1)\oplus L(\mathcal{l},\mathcal{j}_2)\oplus\cdots\oplus L(\mathcal{l},\mathcal{j}_s)$ is completely reducible.
Then the submodule of $M$ generated by $w^\prime$ is a direct sum of certain $L(\mathcal{l},\mathcal{j})$ with admissible weights $\mathcal{j}$,
it is a contradiction. Therefore $W^\prime=M$, i.e., $M$ is a direct sum of $L(\mathcal{l},\mathcal{j})$ with admissible weight $\mathcal{j}$.
Then we have that the category $\mathcal{C}_{\mathcal{l}}$ is semisimple.
\end{proof}

As an application of the semisimplicity of $\mathcal{C}_{\mathcal{l}}$,
we prove the category of ordinary modules for $L_{\widehat{\mathfrak{g}}}(\mathcal{l},0)$ is semisimple.

\begin{prop}\label{thmord}
Any ordinary $L_{\widehat{\mathfrak{g}}}(\mathcal{l},0)$-module $M$ is completely reducible.
\end{prop}
\begin{proof}
Let $M$ be an ordinary $L_{\widehat{\mathfrak{g}}}(\mathcal{l},0)$-module, then
$M=\bigoplus_{h\in\mathbb{C}}M_h$ such that
$\mbox{dim}~M_h<\infty$ for any $h\in\mathbb{C}$ and $M_{h+n}=0$ for $n\in\mathbb{Z}$ sufficiently small.
For any $a(n) \in\mathfrak{g}\otimes t\mathbb{C}[t]$,
we have $\mbox{wt}~a(n)w=\mbox{wt}~w-n<\mbox{wt}~w$ for all $w\in M$, then $a(n)$ acts locally nilpotently on $M$.
For any $w\in M_h~(h\in\mathbb{C})$,
since $M$ is a $\tilde{\mathfrak{g}}$-module
and $\mbox{wt}~aw=\mbox{wt}~w$ for any $a \in\mathfrak{g}$,
we have $M_h$ is a finite dimensional $\mathfrak{g}$-module,
then $x$ and $e$ act locally nilpotently on $M_h$.
Therefore $x$ acts locally nilpotently on $M$ for all $x\in N_+$, i.e., $M$ is an object of $\mathcal{C}_{\mathcal{l}}$,
then $M$ is completely reducible by Theorem \ref{thmcc}.
\end{proof}

\begin{prop}\label{propord}
Any irreducible ordinary $L_{\widehat{\mathfrak{g}}}(\mathcal{l},0)$-module is of the form
$L(\mathcal{l},\mathcal{j})$ such that $0\leq \mathcal{j}\leq \frac{p-2}{2}, \mathcal{j}\in\mathbb{Z}$.
Moreover, up to isomorphism, there are finitely many irreducible ordinary $L_{\widehat{\mathfrak{g}}}(\mathcal{l},0)$-modules.
\end{prop}
\begin{proof}
Let $U$ be an irreducible finite dimensional $\mathfrak{g}$-module,
then the irreducible quotient $L(U)$ of $\mbox{Ind}_{\mathfrak{g}}^{\hat{\mathfrak{g}}}(U)$
is an irreducible ordinary $V(\mathcal{l},\mathbb{C})$-module (cf. \cite{L}).
It is similar to the proof of Theorem 6.2.23 of \cite{LL},
we have the modules $L(U)$ for irreducible finite dimensional $\mathfrak{g}$-module $U$
exhaust the irreducible ordinary $V(\mathcal{l},\mathbb{C})$-modules up to equivalence.
Let $M$ be an irreducible ordinary $L_{\widehat{\mathfrak{g}}}(\mathcal{l},0)$-module,
then from Proposition \ref{propirr} and Proposition \ref{thmord},
$M$ is of the form
$L(\mathcal{l},\mathcal{j})$ with admissible weight $\mathcal{j}$.
Since $M$ is an irreducible ordinary $V(\mathcal{l},\mathbb{C})$-module,
we have $M$ is of the form $L(U)$ for irreducible finite dimensional $\mathfrak{g}$-module $U$.
Hence $M$ is of the form
$L(\mathcal{l},\mathcal{j})$ with admissible weight $\mathcal{j}\in\mathbb{Z}$.

Since $\mathcal{j}$ is admissible, $\mathcal{j}=\frac{m-1}{2}-ls$,
where $l=\frac{p}{2q},m,s\in \mathbb{Z}$
such that $m+s\equiv1(\mbox{mod}~2)$ and $1\leq m\leq p-1, 0\leq s \leq q-1$,
$p,q$ are positive integers such that $p\geq2, p\equiv q(\mbox{mod}~2)$ and $(\frac{p-q}{2},q)=1$.
If $p,q\in 2\mathbb{Z}+1$, we have $(p,q)=1$, $l\notin \mathbb{Z}$, then $s=0$, $\mathcal{j}=\frac{m-1}{2}$ for $1\leq m\leq p-1,m\in 2\mathbb{Z}+1$.
If $p,q\in 2\mathbb{Z}$, we have $\frac{p-q}{2}\in 2\mathbb{Z}+1$ and $(\frac{p}{2},\frac{q}{2})=1$,
for $\frac{p}{2}\in 2\mathbb{Z}+1$, $l=\frac{p/2}{q}\notin \mathbb{Z}$ and $(\frac{p}{2},q)=1$,
then $s=0$, $\mathcal{j}=\frac{m-1}{2}$ for $1\leq m\leq p-1,m\in 2\mathbb{Z}+1$;
for $\frac{p}{2}\in 2\mathbb{Z}$, $l=\frac{p/4}{q/2}$ and $(\frac{p}{4},\frac{q}{2})=1$,
by $m+s\equiv1(\mbox{mod}~2)$,
$\mathcal{j}=\frac{m-1}{2}-\frac{p}{4}\notin \mathbb{Z}$ for $s=\frac{q}{2}$,
then $s=0$, $\mathcal{j}=\frac{m-1}{2}$ for $1\leq m\leq p-1,m\in 2\mathbb{Z}+1$.
Therefore we have $0\leq \mathcal{j}\leq \frac{p-2}{2}, \mathcal{j}\in\mathbb{Z}$.
\end{proof}

Proposition \ref{thmord} and Proposition \ref{propord} are also obtained in \cite{CGL}
by using the theory of vertex superalgebra extensions.

\subsection{$\mathbb{Q}$-graded vertex operator superalgebras $(L_{\widehat{\mathfrak{g}}}(\mathcal{l},0),\omega_\xi)$}
\label{sec:3.3}

A $\mathbb{Q}$-graded vertex operator superalgebra $V$ is $\mathbb{Q}$-graded by weights instead of $\mathbb{Z}$-graded
(cf. \cite{V}).
Similarly we can define the weak module, $\mathbb{Q}_+$-graded weak $V$-module
and ordinary module for $\mathbb{Q}$-graded vertex operator superalgebra.

\begin{defi}
{\em A $\mathbb{Q}$-graded vertex operator superalgebra $V$
is called {\em rational} if any $\mathbb{Q}_+$-graded weak $V$-module is a direct
sum of irreducible $\mathbb{Q}_+$-graded weak $V$-modules.  }
\end{defi}

Let $(V,Y,\textbf{1},\omega)$ be a $\mathbb{Z}$-graded vertex operator superalgebra,
$h\in V_{(1)}$ be a vector satisfying the following conditions:
\begin{equation}\label{eql}
\begin{aligned}
&[L(m),h_n]=-nh_{m+n}-\frac{m^2+m}{2}\delta_{m+n,0}\kappa_1\mbox{Id}_V,\\
&[h_m,h_n]=2m\delta_{m+n,0}\kappa_2\mbox{Id}_V,
\end{aligned}
\end{equation}
where $\kappa_1,\kappa_2\in\mathbb{C}$.
Assume that $h_0$ acts semisimply on $V$ and the eigenvalues of $h_0$ are rational numbers.
For a rational number $\xi$, set $\omega_\xi=\omega+\frac{\xi}{2}L(-1)h$ and $Y(\omega_\xi,z)=\sum_{n\in\mathbb{Z}}L^\prime(n) z^{-n-2}$.
Then we have
$$L^\prime(n)=L(n)-\frac{\xi}{2}(n+1)h_n~~\mbox{for~~any}~~n\in\mathbb{Z},$$
hence the component operators of $\omega_\xi$
satisfy the Virasoro relations and $L^\prime(-1)=L(-1)=D$.
Since $L(0)$ and $h_0$ are commutative,
each $V_{(n)}$ is a direct sum of eigenspaces of $h_0$.
Set $V_{(m,n)}=\{v\in V_{(m)}\mid h_0v=nv\}$ for $m\in\mathbb{Z},n\in \mathbb{Q}$.
Let
\begin{equation}
V^\prime_{(m)}=\{v\in V\mid L^\prime(0)v=mv\}=\coprod_{s\in\mathbb{Z},t\in \mathbb{Q},s-\frac{\xi t}{2}=m}V_{(s,t)}.
\end{equation}
Then we have the following proposition (cf. \cite{AV,DLM2,Lin}).

\begin{prop}\label{propqvosa}
Let $(V,Y,\textbf{1},\omega)$ be a $\mathbb{Z}$-graded vertex operator superalgebra
of central charge $c$, $h\in V_{(1)}$ be a vector satisfying the conditions (\ref{eql})
and $h_0$ acts semisimply on $V$, the eigenvalues of $h_0$ are rational numbers.
Suppose that $\mbox{dim}~V^\prime_{(m)}<\infty$ for any $m\in\mathbb{Q}$
and $V^\prime_{(m)}=0$ for $m$ sufficiently small.
Then $(V,Y,\textbf{1},\omega_\xi)$ is a $\mathbb{Q}$-graded vertex operator superalgebra of central charge $c-6\xi(\kappa_1+\xi\kappa_2)$.
\end{prop}

For the $\mathbb{Z}$-graded vertex operator superalgebra $L_{\widehat{\mathfrak{g}}}(\mathcal{l},0)$,
by the Sugawara construction,
\begin{equation*}
\begin{aligned}
L(n)=&\frac{1}{2\mathcal{l}+3}\sum_{m\in\mathbb{Z}}(\frac{1}{2}:h(m)h(n-m):+:e(m)f(n-m):+:f(m)e(n-m):\\
&-\frac{1}{2}:x(m)y(n-m):+\frac{1}{2}:y(m)x(n-m):).
\end{aligned}
\end{equation*}
A direct calculation shows that
$[L(m),h(n)]=-nh(m+n)$, then $h$ satisfies conditions (\ref{eql})
with $\kappa_1=0, \kappa_2=\mathcal{l}$.
Hence $L_{\widehat{\mathfrak{g}}}(\mathcal{l},0)$ have a $\mathbb{Q}$-gradation by weights with respect to $L^\prime(0)$.
From the construction of $V(\mathcal{l},\mathbb{C})$,
for any $m\in\mathbb{Z}_+$, if $V(\mathcal{l},\mathbb{C})_{(m,n)}\ne 0$,
we have $-2m\leq n\leq 2m$.
Since $L_{\widehat{\mathfrak{g}}}(\mathcal{l},0)$ is a quotient of $V(\mathcal{l},\mathbb{C})$,
for any $m\in\mathbb{Z}_+$, if $L_{\widehat{\mathfrak{g}}}(\mathcal{l},0)_{(m,n)}\ne 0$,
we also have $-2m\leq n\leq 2m$.
Let $0<\xi<1$, for any $m\in\mathbb{Q}$,
since $L_{\widehat{\mathfrak{g}}}(\mathcal{l},0)^\prime_{(m)}=\coprod_{s\in\mathbb{Z},t\in \mathbb{Q},s-\frac{\xi t}{2}=m}L_{\widehat{\mathfrak{g}}}(\mathcal{l},0)_{(s,t)}$,
we have $$\frac{m}{1+\xi}\leq s\leq\frac{m}{1-\xi}$$
for $L_{\widehat{\mathfrak{g}}}(\mathcal{l},0)_{(s,t)}\ne0$,
then $\mbox{dim}~L_{\widehat{\mathfrak{g}}}(\mathcal{l},0)^\prime_{(m)}<\infty$.
Note that $\frac{m}{1+\xi}\leq \frac{m}{1-\xi}$,
we have $L_{\widehat{\mathfrak{g}}}(\mathcal{l},0)^\prime_{(m)}=0$ for $m<0$.
Therefore, by Proposition \ref{propqvosa},
$(L_{\widehat{\mathfrak{g}}}(\mathcal{l},0),\omega_\xi)$ is a $\mathbb{Q}$-graded vertex operator superalgebra
of central charge $c-6\xi^2\mathcal{l}$.

\begin{thm}
Let $0<\xi<1$, $(L_{\widehat{\mathfrak{g}}}(\mathcal{l},0),\omega_\xi)$ is a rational $\mathbb{Q}$-graded vertex operator superalgebra.
Moreover, any $\mathbb{Q}_+$-graded weak $(L_{\widehat{\mathfrak{g}}}(\mathcal{l},0),\omega_\xi)$-module is a direct
sum of $L(\mathcal{l},\mathcal{j})$ with admissible weight $\mathcal{j}$.
\end{thm}
\begin{proof}
Let $M$ be a $\mathbb{Q}_+$-graded weak $(L_{\widehat{\mathfrak{g}}}(\mathcal{l},0),\omega_\xi)$-module,
we denote by $\mbox{deg}~w$ the degree of homogeneous element $w\in W$.
For any $\alpha\in \Delta_{\mathfrak{g}}^+$ and the corresponding root vector $x_\alpha$,
 $$L^\prime(0)x_\alpha=L(0)x_\alpha-\frac{\xi}{2}h_0x_\alpha=(1-\frac{\xi}{2}\alpha(h))x_\alpha.$$
 Since $0<\xi<1$, we have
$\mbox{deg}~x_\alpha w=\mbox{wt}~x_\alpha+\mbox{deg}~w-1<\mbox{deg}~w$ for all $w\in M$,
then $x$ and $e$ act locally nilpotently on $M$.
For any $a(n) \in\mathfrak{g}\otimes t\mathbb{C}[t]$, we have
$$\mbox{deg}~a(n)w=\mbox{deg}~w+\mbox{wt}~a-n-1<\mbox{deg}~w$$
for all $w\in M$, then $a(n)$ acts locally nilpotently on $M$.
Therefore $a$ acts locally nilpotently on $M$ for all $a\in N_+$, i.e., $M$ is an object of $\mathcal{C}_{\mathcal{l}}$.
Then $M$ is completely reducible by Theorem \ref{thmcc}.

Next we prove that $L(\mathcal{l},\mathcal{j})$ is a $\mathbb{Q}_+$-graded weak
$(L_{\widehat{\mathfrak{g}}}(\mathcal{l},0),\omega_\xi)$-module for any admissible weight $j$.
Similar to the proof that $(L_{\widehat{\mathfrak{g}}}(\mathcal{l},0),\omega_\xi)$ is a $\mathbb{Q}$-graded vertex operator superalgebra,
$L(\mathcal{l},\mathcal{j})$ with admissible weight $\mathcal{j}$ is an ordinary $(L_{\widehat{\mathfrak{g}}}(\mathcal{l},0),\omega_\xi)$-module,
hence $L(\mathcal{l},\mathcal{j})$ is a $\mathbb{Q}_+$-graded weak
$(L_{\widehat{\mathfrak{g}}}(\mathcal{l},0),\omega_\xi)$-module.
\end{proof}

\section{$A(V)$-theory for $(L_{\widehat{\mathfrak{g}}}(\mathcal{l},0),\omega_\xi)$}
\label{sec:4}
	\def\theequation{4.\arabic{equation}}
	\setcounter{equation}{0}

In this section,
we first recall the $A(V)$-theory for $\mathbb{Q}$-graded vertex operator superalgebras.
Then we determine the Zhu's algebras $A(L_{\widehat{\mathfrak{g}}}(\mathcal{l},0))$ and their bimodules $A(L(\mathcal{l},\mathcal{j}))$
for $\mathbb{Q}$-graded vertex operator superalgebras $(L_{\widehat{\mathfrak{g}}}(\mathcal{l},0),\omega_\xi)$.
Furthermore, we apply the Zhu's algebras and their bimodules associated to $(L_{\widehat{\mathfrak{g}}}(\mathcal{l},0),\omega_\xi)$
to calculate the fusion rules among the irreducible ordinary modules of $(L_{\widehat{\mathfrak{g}}}(\mathcal{l},0),\omega_\xi)$.

\subsection{$A(V)$-theory for $\mathbb{Q}$-graded vertex operator superalgebra}

Let $V$ be a $\mathbb{Q}$-graded vertex operator superalgebra.
Define a function $\varepsilon$ for all homogeneous elements of $V$ as follows:
\begin{equation*}
\varepsilon(a)=\begin{cases}
1,& \mbox{wt}~a\in\mathbb{Z},\\
0,& \mbox{wt}~a\notin\mathbb{Z}.\\
\end{cases}
\end{equation*}
For any homogeneous element $a\in V$, we define
\begin{equation}
a\ast b=\varepsilon(a)\mbox{Res}_z\frac{(1+z)^{[{\rm wt}~a]}}{z}Y(a,z)b
\end{equation}
for any $b\in V$, where $[\cdot]$ denotes the greatest-integer function.
Then we can extend $\ast$ on $V$.
Let $O(V)$ be the subspace of $V$ linearly spanned by
\begin{equation}
\mbox{Res}_z\frac{(1+z)^{[{\rm wt}~a]}}{z^{1+\varepsilon(a)}}Y(a,z)b
\end{equation}
for any homogeneous element $a\in V$ and for any $b\in V$.
We have
\begin{equation}
\mbox{Res}_z\frac{(1+z)^{[{\rm wt}~a]+m}}{z^{1+\varepsilon(a)+n}}Y(a,z)b\in O(V)
\end{equation}
for $n\geq m\geq0$ (cf. \cite{FZ}).
Let $M$ be any weak $V$-module, we define
\begin{equation}\nonumber
\Omega(M)=\{u\in M\mid a_m u=0~~\mbox{for}~~a\in V, m>\mbox{wt}~a-1\},
\end{equation}
and $o$ to be the linear map from $V$ to $\mbox{End}~\Omega(M)$ such that $o(a)=\varepsilon(a)a_{[{\rm wt}~a]-1}$
for any homogeneous element $a\in V$.
The following theorem were established in \cite{DK}.
\begin{thm}\label{thmzhu}
{\rm (a)} The subspace $O(V)$ is a two-sided ideal of $V$ with respect to the product $\ast$
and $A(V)=V/O(V)$ is an associative superalgebra with identity $\textbf{1}+O(V)$.
Moreover, $\omega+O(V)$ lies in the center of $A(V)$.\\
{\rm (b)} For any weak $V$-module $M$,
$\Omega(M)$ is an $A(V)$-module with $a$ acts as $o(a)$ for $a\in V$.\\
{\rm (c)} There is an induction functor $L$ from the category of $A(V)$-modules to the category of $\mathbb{Q}_+$-graded weak $V$-modules.
Moreover, $\Omega(L(U))=U$ for any $A(V)$-module $U$.\\
{\rm (d)} $\Omega$ and $L$ are inverse bijections between the sets of irreducible modules in each category.\\
\end{thm}
\vspace{-0.5cm}
As a consequence, we have the following proposition.
\begin{prop}\label{prop9}
If $V$ is rational, then $A(V)$ is a semisimple associative superalgebra.
\end{prop}
\begin{proof}
Let $U$ be a $A(V)$-module, from Theorem \ref{thmzhu} $L(U)$ is a $\mathbb{Q}_+$-graded weak $V$-modules with $\Omega(L(U))=L(U)(0)=U$.
Since $V$ is rational, we have $L(U)$ is complete reducible.
Then $U$ is complete reducible as $A(V)$-module.
Hence $A(V)$ is semisimple.
\end{proof}

For any weak $V$-module $M$, let $O(M)$ be the subspace of $M$ linearly spanned by
\begin{equation}
\mbox{Res}_z\frac{(1+z)^{[{\rm wt}~a]}}{z^{1+\varepsilon(a)}}Y_M(a,z)u
\end{equation}
for any homogeneous element $a\in V$ and for any $u\in M$.
Similar to the Theorem 2.11 in \cite{Li2} (see also \cite{FZ,KWa}),
we have the following theorem.

\begin{thm}\label{thmam}
{\rm (a)} The quotient space $A(M)=M/O(M)$ is an $A(V)$-bimodule with the following left and right actions:
\begin{equation}
\begin{aligned}
&a\ast u=\varepsilon(a){\rm Res}_z\frac{(1+z)^{[{\rm wt}~a]}}{z}Y_M(a,z)u,\\
&u\ast a=(-1)^{|a||u|}\varepsilon(a){\rm Res}_z\frac{(1+z)^{[{\rm wt}~a]-1}}{z}Y_M(a,z)u
\end{aligned}
\end{equation}
for any homogeneous element $a\in V, u\in M$.\\
{\rm (b)} Let $W_1, W_2, W_3$ be irreducible $V$-modules and suppose $V$ is rational.
Then there is a linear isomorphism from the space ${\rm Hom}_{A(V)}(A(W_1)\otimes_{A(V)}W_2(0),W_3(0))$
to the space of intertwining operators of type
$\begin{pmatrix}W_3\\ W_1~~W_2\end{pmatrix}$.
\end{thm}

We also obtain the following lemma which is similar to the Proposition 1.5.4 in \cite{FZ}.

\begin{lem}\label{lemfz}
Let $V$ be a $\mathbb{Q}$-graded vertex operator superalgebra
and let $M$ be a weak $V$-module with a weak submodule $W$, $I$ be an ideal of $V$.\\
{\rm (a)} As $A(V)$-bimodule $A(M/W)\cong M/(O(M)+W)$.\\
{\rm (b)} $(I+O(V))/O(V)$ is a two sided ideal of $A(V)$ and $A(V/I)\cong A(V)/((I+O(V))/O(V))$.\\
{\rm (c)} If $I\cdot M\subset W$, then $I\ast A(M)\subset(W+O(M))/O(M)$, $A(M)\ast I\subset(W+O(M))/O(M)$
and $A(M/W)\cong A(M)/((W+O(M))/O(M))$ as $A(V/I)$-bimodules.
\end{lem}

\subsection{Zhu's algebras $A(L_{\widehat{\mathfrak{g}}}(\mathcal{l},0))$ and their bimodules $A(L(\mathcal{l},\mathcal{j}))$}
\label{sec:4.2}

Take $0<\xi<1$, we have $(L_{\widehat{\mathfrak{g}}}(\mathcal{l},0),\omega_\xi)$
and $(V(\mathcal{l},\mathbb{C}),\omega_\xi)$
are $\mathbb{Q}$-graded vertex operator superalgebras.
First we determine the Zhu's algebra $A(V(\mathcal{l},\mathbb{C}))$ and their bimodules $A(M(\mathcal{l},\mathcal{j}))$
for $(V(\mathcal{l},\mathbb{C}),\omega_\xi)$.

\begin{prop}\label{propom}
Let $M$ be any weak $(V(\mathcal{l},\mathbb{C}),\omega_\xi)$-module.
Then $O(M)=CM$, where
\begin{equation*}
C=\mathbb{C}[t^{-1}](t^{-1}+1)\otimes f+\mathbb{C}[t^{-1}](t^{-1}+1)\otimes y+\mathbb{C}[t^{-1}]t^{-1}\otimes e
+\mathbb{C}[t^{-1}]t^{-1}\otimes x+\mathbb{C}[t^{-1}](t^{-2}+t^{-1})\otimes h.
\end{equation*}
\end{prop}
\begin{proof}
Since
\begin{equation*}
\begin{aligned}
  &\mbox{wt}~h(-1)\textbf{1}=1, \mbox{wt}~e(-1)\textbf{1}=1-\xi, \mbox{wt}~f(-1)\textbf{1}=1+\xi,\\ &\mbox{wt}~x(-1)\textbf{1}=1-\frac{1}{2}\xi,\mbox{wt}~y(-1)\textbf{1}=1+\frac{1}{2}\xi,
\end{aligned}
\end{equation*}
we have
\begin{equation}\label{fyexh}
\begin{aligned}
&\mbox{Res}_z\frac{(1+z)^{[{\rm wt}~f]}}{z^{m}}Y_M(f,z)u=(f(-m)+f(1-m))u;\\
&\mbox{Res}_z\frac{(1+z)^{[{\rm wt}~y]}}{z^{m}}Y_M(y,z)u=(y(-m)+y(1-m))u;\\
&\mbox{Res}_z\frac{(1+z)^{[{\rm wt}~e]}}{z^{m}}Y_M(e,z)u=e(-m)u;\\
&\mbox{Res}_z\frac{(1+z)^{[{\rm wt}~x]}}{z^{m}}Y_M(x,z)u=x(-m)u;\\
&\mbox{Res}_z\frac{(1+z)^{[{\rm wt}~h]}}{z^{m+1}}Y_M(h,z)u=(h(-m-1)+h(-m))u
\end{aligned}
\end{equation}
for any positive integer $m$ and for $u\in M$.
It is clear that
all those elements in (\ref{fyexh}) are in $O(M)$.
Let $W$ be the subspace linearly spanned by the elements in (\ref{fyexh}),
then $W=CM$.

Let $L$ be the linear span of homogeneous elements $a$ of $(V(\mathcal{l},\mathbb{C}),\omega_\xi)$
such that for any positive integer $n$,
\begin{equation*}
\mbox{Res}_z\frac{(1+z)^{[{\rm wt}~a]}}{z^{n+\varepsilon(a)}}Y_M(a,z)M\subseteq W.
\end{equation*}
Let $a$ be any homogeneous element of $L$,
similar to the proof of the Proposition 4.1 in \cite{DLM2},
we have $h(-m)a,e(-m)a,f(-m)a,x(-m)a,y(-m)a\in L$ for any positive integer $m$.
Since $\textbf{1}\in L$ and $V(\mathcal{l},\mathbb{C})=U(t^{-1}\mathbb{C}[t^{-1}]\otimes \mathfrak{g})\textbf{1}$,
we have $L=V(\mathcal{l},\mathbb{C})$.
Therefore $O(M)=W=CM$.
\end{proof}

\begin{prop}\label{propav}
The associative superalgebra $A(V(\mathcal{l},\mathbb{C}))$ for
$(V(\mathcal{l},\mathbb{C}),\omega_\xi)$ is isomorphic to the polynomial algebra $\mathbb{C}[t]$.
\end{prop}
\begin{proof}
Define a linear map
\begin{equation*}
 \begin{aligned}
 \psi:&\mathbb{C}[t]\rightarrow A(V(\mathcal{l},\mathbb{C}))\\
 &g(t)\mapsto g(h(-1))\textbf{1}+O(V(\mathcal{l},\mathbb{C}))
 \end{aligned}
\end{equation*}
for any $g(t)\in \mathbb{C}[t]$.
Since $[h(-1),h(0)]=0$ and $h(0)\textbf{1}=0$,
we have $g(h(-1))\textbf{1}=g(h(-1)+h(0))\textbf{1}$ for any $g(t)\in \mathbb{C}[t]$.
Since $h(-1)\textbf{1}\ast u=(h(-1)+h(0))u$ for any $u\in V(\mathcal{l},\mathbb{C})$, we have
\begin{align*}
  \psi(f(t)g(t))&=f(h(-1))g(h(-1))\textbf{1}+O(V(\mathcal{l},\mathbb{C}))\\
  &=f(h(-1)+h(0))g(h(-1))\textbf{1}+O(V(\mathcal{l},\mathbb{C}))\\
  &=f(h(-1))\textbf{1}\ast g(h(-1))\textbf{1}+O(V(\mathcal{l},\mathbb{C}))=\psi(f(t))\ast\psi(g(t)),
\end{align*}
thus $\psi$ is an superalgebra homomorphism.
By $N_-=C\oplus\mathbb{C}f(0)\oplus\mathbb{C}y(0)\oplus\mathbb{C}h(-1)$,
$$U(N_-)=U(C)U(\mathbb{C}h(-1))U(\mathbb{C}f(0))U(\mathbb{C}y(0)).$$
Then by Proposition \ref{propom}, we have
\begin{equation*}
  O(V(\mathcal{l},\mathbb{C}))=CV(\mathcal{l},\mathbb{C})=CU(N_-)\textbf{1}\cong CU(C)U(\mathbb{C}h(-1)).
\end{equation*}
Hence
$A(V(\mathcal{l},\mathbb{C}))\cong (U(C)U(\mathbb{C}h(-1)))/(CU(C)U(\mathbb{C}h(-1)))\cong U(\mathbb{C}h(-1))\cong \mathbb{C}[t]$.
\end{proof}

\begin{prop}\label{propavam}
The $A(V(\mathcal{l},\mathbb{C}))$-bimodule $A(M(\mathcal{l},\mathcal{j}))$ is isomorphic to
$\mathbb{C}[t_1,t_2]$ with the biaction as follows:
\begin{equation}
t\ast f(t_1,t_2)=(t_1+\mathcal{j}-t_2\frac{\partial}{\partial t_2})f(t_1,t_2), f(t_1,t_2)\ast t=t_1f(t_1,t_2)
\end{equation}
for any $f(t_1,t_2)\in\mathbb{C}[t_1,t_2]$.
\end{prop}
\begin{proof}
Let $v$ be a highest weight vector of $M(\mathcal{l},\mathcal{j})$.
Then by Proposition \ref{propom}, we have
\begin{equation*}
  O(M(\mathcal{l},\mathcal{j}))=CM(\mathcal{l},\mathcal{j})=CU(N_-)v\cong CU(C)U(\mathbb{C}h(-1))U(\mathbb{C}f(0))U(\mathbb{C}y(0)).
\end{equation*}
Then
\begin{equation*}
  A(M(\mathcal{l},\mathcal{j}))=\bigoplus_{m_1,m_2\in \mathbb{Z}_+}\mathbb{C}(h(-1)^{m_1}y(0)^{m_2}+ O(M(\mathcal{l},\mathcal{j}))).
\end{equation*}
From Theorem \ref{thmam}, we have
\begin{align*}
&\begin{aligned}
  h(-1)\textbf{1}\ast (h(-1)^{m_1}y(0)^{m_2}v)&=(h(-1)+h(0))h(-1)^{m_1}y(0)^{m_2}v\\
  &=(h(-1)+\mathcal{j}-m_2)h(-1)^{m_1}y(0)^{m_2}v,\\
  \end{aligned} \\
  &(h(-1)^{m_1}y(0)^{m_2}v)\ast h(-1)\textbf{1}=h(-1)h(-1)^{m_1}y(0)^{m_2}v,
\end{align*}
for any $m_1,m_2\in \mathbb{Z}_+$. Define a linear map
\begin{equation*}
 \begin{aligned}
 \varphi:&\mathbb{C}[t_1,t_2]\rightarrow A(M(\mathcal{l},\mathcal{j}))\\
 &t_1^{m_1}t_2^{m_2}\mapsto h(-1)^{m_1}y(0)^{m_2}+O(M(\mathcal{l},\mathcal{j}))~~~~\forall m_1,m_2\in \mathbb{Z}_+,\\
 \end{aligned}
\end{equation*}
it is clear $\varphi$ is a bimodule isomorphic.
\end{proof}

Now we determine the Zhu's algebras $A(L_{\widehat{\mathfrak{g}}}(\mathcal{l},0))$ and their bimodules $A(L(\mathcal{l},\mathcal{j}))$
for $(L_{\widehat{\mathfrak{g}}}(\mathcal{l},0),\omega_\xi)$.

\begin{thm}\label{thmav}
The associative superalgebra $A(L_{\widehat{\mathfrak{g}}}(\mathcal{l},0))$ for
$(L_{\widehat{\mathfrak{g}}}(\mathcal{l},0),\omega_\xi)$
is semisimple and isomorphic to the quotient algebra $\mathbb{C}[t]/\langle f(t)\rangle$ of the polynomial algebra $\mathbb{C}[t]$,
where
\begin{equation}
f(t)=\prod_{i+j\in2\mathbb{Z}+1,0\leq j\leq q-1,1\leq i\leq p-1}(t-\frac{i-1}{2}+jl).
\end{equation}
\end{thm}
\begin{proof}
Since $(L_{\widehat{\mathfrak{g}}}(\mathcal{l},0),\omega_\xi)$ is a quotient
vertex operator superalgebra of $(V(\mathcal{l},\mathbb{C}),\omega_\xi)$,
then by Lemma \ref{lemfz}, $A(L_{\widehat{\mathfrak{g}}}(\mathcal{l},0))$ is
quotient superalgebra of $A(V(\mathcal{l},\mathbb{C}))\cong \mathbb{C}[t]$.
Set $A(L_{\widehat{\mathfrak{g}}}(\mathcal{l},0))=A(V(\mathcal{l},\mathbb{C}))/I$.
By the proof of Proposition \ref{propav}, any irreducible $A(V(\mathcal{l},\mathbb{C}))$-module $U$ is also an irreducible $\mathbb{C}h$-module.
Then for any irreducible $A(V(\mathcal{l},\mathbb{C}))$-module $U$, by the definition of $L(\mathcal{l}, U)$,
$L(\mathcal{l}, U)$ is the corresponding irreducible $\mathbb{Q}_+$-graded weak $V(\mathcal{l},\mathbb{C})$-module
in Theorem \ref{thmzhu}.
From Theorem \ref{thm1},
$L(\mathcal{l},U)$ is an irreducible $\mathbb{Q}_+$-graded weak $L_{\widehat{\mathfrak{g}}}(\mathcal{l},0)$-module
if and only if $f(h)U=0$ and $U$ is an irreducible $\mathbb{C}h$-module.
$L(\mathcal{l},U)$ is an irreducible $\mathbb{Q}_+$-graded weak $L_{\widehat{\mathfrak{g}}}(\mathcal{l},0)$-module
is equivalent to $U$ is an irreducible $A(L_{\widehat{\mathfrak{g}}}(\mathcal{l},0))$-module, i.e., $I.U=0$.
Then we have $I=\langle f(t)\rangle$.
Since $(L_{\widehat{\mathfrak{g}}}(\mathcal{l},0),\omega_\xi)$ is rational, by Proposition \ref{prop9},
$A(L_{\widehat{\mathfrak{g}}}(\mathcal{l},0))$ is semisimple.
\end{proof}

Let $B_0=\mathbb{C}(t^{-1}+1)\otimes f+\mathbb{C}(t^{-1}+1)\otimes y+(t^{-2}+t^{-1})\mathbb{C}[x^{-1}]\otimes \mathfrak{g}$.
Then $B_0$ is an ideal of $N_-$, we denote
$L_0=N_-/B_0=\mathbb{C}T_e+\mathbb{C}T_f+\mathbb{C}T_h+\mathbb{C}T_x+\mathbb{C}T_y,$
where $T_e=e(-1)+B_0,T_f=f(0)+B_0,T_h=h(-1)+B_0,T_x=x(-1)+B_0,T_y=y(0)+B_0$
satisfies the following anticommutation relations:
\begin{equation}\nonumber
\begin{aligned}
&[T_e,T_f]=T_h,~[T_h,T_e]=-2T_e,~[T_h,T_f]=2T_f,\\
&[T_h,T_x]=-T_x,~[T_e,T_x]=0,~[T_f,T_x]=T_y,\\
&[T_h,T_y]=T_y,~[T_e,T_y]=-T_x,~[T_f,T_y]=0,\\
&\{T_x,T_x\}=-2T_e,\{T_x,T_y\}=T_h,\{T_y,T_y\}=-2T_f.
\end{aligned}
\end{equation}
Let $\pi_1$ be the natural quotient map from $U(N_-)$ onto $U(L_0)$.
Set $P_1(\alpha)=T_xT_y-\alpha$ and $Q_1(\alpha)=T_yT_x+\alpha$ for $\alpha\in\mathbb{C}$.
We have the following commutation relations in $U(L_0)$.

\begin{prop}\label{propl0}
For any $\alpha,\beta,\gamma\in\mathbb{C}$, we have
\begin{align*}
&[P_1(\alpha),P_1(\beta)]=0,[P_1(\alpha),Q_1(\beta)]=0,[Q_1(\alpha),Q_1(\beta)]=0,\\
&{T_e}^\gamma P_1(\alpha)=P_1(\alpha-\gamma){T_e}^\gamma, {T_e}^\gamma Q_1(\alpha)=Q_1(\alpha+\gamma){T_e}^\gamma,\\
&{T_f}^\gamma P_1(\alpha)=P_1(\alpha+\gamma){T_f}^\gamma, {T_f}^\gamma Q_1(\alpha)=Q_1(\alpha-\gamma){T_f}^\gamma,\\
&{T_x}P_1(\alpha)=Q_1(1-\alpha){T_x}, {T_x}Q_1(\alpha)=P_1(-\alpha){T_x}, \\
&T_y P_1(\alpha)=Q_1(-\alpha)T_y, T_y Q_1(\alpha)=P_1(1-\alpha)T_y,\\
&{T_e}^\gamma T_y=-Q_1(\gamma)T_x{T_e}^{\gamma-1},T_x{T_e}^\gamma T_y=P_1(-\gamma){T_e}^\gamma,\\
&{T_f}^\gamma T_x=-P_1(\gamma)T_y{T_f}^{\gamma-1},T_y{T_f}^\gamma T_x=Q_1(-\gamma){T_f}^\gamma.
\end{align*}
\end{prop}

Similar to the Lemma 4.2 in \cite{IK}, we have the following proposition by induction.

\begin{prop}\label{propl01}
The following projection formulas hold:
\begin{equation*}
\begin{aligned}
&\pi_1(F_1(m,s))=\theta_{(m,s)}\prod_{j=1}^s\prod_{i=1}^m\{\prod_{i+j\in2\mathbb{Z}}P_1(\frac{i}{2}+jl)\prod_{i+j\in2\mathbb{Z}+1}Q_1(-\frac{i-1}{2}-jl)\}{T_y}^m,\\
&\pi_2(F_2(m,s))=\theta_{(p-m,q-s)}\prod_{j=1}^{q-s-1}\prod_{i=1}^{p-m}\{\prod_{i+j\in2\mathbb{Z}}Q_1(\frac{i}{2}+jl)\prod_{i+j\in2\mathbb{Z}+1}P_1(-\frac{i-1}{2}-jl)\}{T_x}^{p-m},
\end{aligned}
\end{equation*}
where $\theta_{(m,s)}=(-1)^{\frac{m(1+(-1)^m)+s(1+(-1)^s)}{4}}$.
\end{prop}
\begin{thm}\label{prop510}
Let $\mathcal{j}=\frac{m-1}{2}-ls$ be an admissible weight.
Then the $A(L_{\widehat{\mathfrak{g}}}(\mathcal{l},0))$-bimodule $A(L(\mathcal{l},\mathcal{j}))$ is isomorphic to the quotient space of $\mathbb{C}[t_1,t_2]$ modulo the subspace
\begin{equation}
\mathbb{C}[t_1,t_2]t_2^m+\mathbb{C}[t_1]f_{\mathcal{j},0}(t_1,t_2)+\mathbb{C}[t_1]f_{\mathcal{j},1}(t_1,t_2)+\cdots+\mathbb{C}[t_1]f_{\mathcal{j},m-1}(t_1,t_2)
\end{equation}
where $$f_{\mathcal{j},n}(t_1,t_2)= (\prod_{i+n+j\in2\mathbb{Z},0\leq j\leq q-s-1, 0\leq i\leq p-m-1}(t_1-\frac{i+n}{2}+jl))t_2^{n}.$$ The left and right actions of $A(L_{\widehat{\mathfrak{g}}}(\mathcal{l},0))$ on $A(L(\mathcal{l},\mathcal{j}))$ are given by Proposition \ref{propavam}.
\end{thm}
\begin{proof}
Note that $C=B_0\oplus \mathbb{C}e(-1)\oplus \mathbb{C}x(-1)$.
Then from Proposition \ref{propom}, we have
\begin{equation*}
O(M(\mathcal{l},\mathcal{j}))=CM(\mathcal{l},\mathcal{j})\cong B_0U(N_-)+e(-1)U(N_-)+x(-1)U(N_-).
\end{equation*}
Since $B_0$ is an ideal of $N_-$, $U(N_-)B_0=B_0U(N_-)$ is an ideal of $U(N_-)$.
From Proposition \ref{propsv}, the maximal proper submodule of $M(\mathcal{l},\mathcal{j})$ is generated by $v_{\mathcal{j},1}$ and $v_{\mathcal{j},2}$.
Then by Lemma \ref{lemfz}, we have
\begin{align*}
A(L(\mathcal{l},\mathcal{j}))&\cong M(\mathcal{l},\mathcal{j})/(CM(\mathcal{l},\mathcal{j})+U(N_-)v_{\mathcal{j},1}+U(N_-)v_{\mathcal{j},2})\\
&\begin{aligned}
\cong &U(N_-)/(B_0U(N_-)+e(-1)U(N_-)+x(-1)U(N_-)+\\&U(N_-)F_1(m,s)+U(N_-)F_2(m,s))
\end{aligned}
\end{align*}
as $A(L(\mathcal{l},0))$-bimodules. Note that $U(N_-)/B_0U(N_-)\cong U(L_0)$. Thus
\begin{equation*}
A(L(\mathcal{l},\mathcal{j}))\cong U(L_0)/(U(L_0)\pi_1(F_1(m,s))+U(L_0)\pi_1(F_2(m,s))+T_xU(L_0)).
\end{equation*}

For any $a,b,d\in \mathbb{Z}_+$, by Proposition \ref{propl0} and Proposition \ref{propl01}, we have
\begin{align*}
&T_x^aT_h^bT_y^d\pi_1(F_1(m,s))\\
&=\theta_{(m,s)}T_x^aT_h^bT_y^d\prod_{j=1}^s\prod_{i=1}^m\{\prod_{i+j\in2\mathbb{Z}}P_1(\frac{i}{2}+jl)\prod_{i+j\in2\mathbb{Z}+1}Q_1(-\frac{i-1}{2}-jl)\}{T_y}^m\\
&=\begin{cases}
\theta_{(m,s)}T_x^a\prod_{j=1}^s\prod_{i=1}^m\{\prod_{i+j\in2\mathbb{Z}}Q_1(\frac{1-i-d}{2}-jl)\prod_{i+j\in2\mathbb{Z}+1}P_1(\frac{i+d}{2}+jl)\}T_h^b{T_y}^{d+m},
~d\notin 2\mathbb{Z}\\
\theta_{(m,s)}T_x^a\prod_{j=1}^s\prod_{i=1}^m\{\prod_{i+j\in2\mathbb{Z}}P_1(\frac{i+d}{2}+jl)\prod_{i+j\in2\mathbb{Z}+1}Q_1(\frac{1-i-d}{2}-jl)\}T_h^b{T_y}^{d+m},
~d\in 2\mathbb{Z}
\end{cases}\\
&=\theta_{(m,s)}T_x^a\prod_{j=1}^s\prod_{i=d+1}^{m+d}\{\prod_{i+j\in2\mathbb{Z}+1}Q_1(-\frac{i-1}{2}-jl)\prod_{i+j\in2\mathbb{Z}}P_1(\frac{i}{2}+jl)\}T_h^b{T_y}^{d+m}\\
&\equiv
\theta_{(m,s)}T_x^a\prod_{j=1}^s\prod_{i=d+1}^{m+d}\{\prod_{i+j\in2\mathbb{Z}+1}(T_h-\frac{i-1}{2}-jl)\prod_{i+j\in2\mathbb{Z}}(-\frac{i}{2}-jl)\}T_h^b{T_y}^{d+m}
~~(\mbox{mod}~T_xU(L_0)).
\end{align*}
Note that $-\frac{i}{2}-jl\ne 0$ for $i+j\in2\mathbb{Z},1\leq j\leq s, d+1\leq i\leq m+d$, we have
\begin{align*}
&U(L_0)\pi_1(F_1(m,s))+T_xU(L_0)\\
=&T_xU(L_0)+\sum_{d=0}^{\infty}\mathbb{C}[T_h](\prod_{i+d+j\in2\mathbb{Z}+1,1\leq j\leq s, 1\leq i\leq m}(T_h-\frac{i+d-1}{2}-jl)){T_y}^{d+m}.
\end{align*}
Similarly, let $a,b,d\in \mathbb{Z}_+$, if $d<p-m$, we have
\begin{align*}
&T_x^aT_h^bT_y^d\pi_1(F_2(m,s))\\
&=\theta_{(p-m,q-s)}T_x^aT_h^bT_y^d\prod_{j=1}^{q-s-1}\prod_{i=1}^{p-m}\{\prod_{i+j\in2\mathbb{Z}}Q_1(\frac{i}{2}+jl)\prod_{i+j\in2\mathbb{Z}+1}P_1(-\frac{i-1}{2}-jl)\}{T_x}^{p-m}\\
&=\theta_{(p-m,q-s)}T_x^aT_h^bT_y^d{T_x}^{p-m}\prod_{j=1}^{q-s-1}\prod_{i=0}^{p-m-1}\{\prod_{i+j\in2\mathbb{Z}}Q_1(-\frac{i}{2}+jl)\prod_{i+j\in2\mathbb{Z}+1}P_1(\frac{i+1}{2}-jl)\}\\
&\begin{aligned}=&\theta_{(p-m,q-s)}T_x^aT_h^b\prod_{i=1}^{d}\{\prod_{i\in 2\mathbb{Z}}P_1(\frac{i}{2})\prod_{i\in 2\mathbb{Z}+1}Q_1(-\frac{i-1}{2})\}
{T_x}^{p-m-d}\prod_{j=1}^{q-s-1}\prod_{i=0}^{p-m-1}\\&\{\prod_{i+j\in2\mathbb{Z}}Q_1(-\frac{i}{2}+jl)\prod_{i+j\in2\mathbb{Z}+1}P_1(\frac{i+1}{2}-jl)\}\end{aligned}\\
&\begin{aligned}=&\theta_{(p-m,q-s)}T_x^aT_h^b{T_x}^{p-m-d}\prod_{i=1+p-m-d}^{p-m}\{\prod_{i\in 2\mathbb{Z}}P_1(\frac{i}{2})\prod_{i\in 2\mathbb{Z}+1}Q_1(-\frac{i-1}{2})\}\\&
\prod_{j=1}^{q-s-1}\prod_{i=0}^{p-m-1}\{\prod_{i+j\in2\mathbb{Z}}Q_1(-\frac{i}{2}+jl)\prod_{i+j\in2\mathbb{Z}+1}P_1(\frac{i+1}{2}-jl)\}\end{aligned}\\
&\begin{aligned}=&\theta_{(p-m,q-s)}T_x^{a+p-m-d}(T_h-p+m+d)^b\prod_{i=1+p-m-d}^{p-m}\{\prod_{i\in 2\mathbb{Z}}P_1(\frac{i}{2})\prod_{i\in 2\mathbb{Z}+1}Q_1(-\frac{i-1}{2})\}\\&
\prod_{j=1}^{q-s-1}\prod_{i=0}^{p-m-1}\{\prod_{i+j\in2\mathbb{Z}}Q_1(-\frac{i}{2}+jl)\prod_{i+j\in2\mathbb{Z}+1}P_1(\frac{i+1}{2}-jl)\}\end{aligned}\\
&\equiv 0~~(\mbox{mod}~T_xU(L_0)).
\end{align*}
If $d=n+p-m$ for some $n\in\mathbb{Z}_+$, we have
\begin{align*}
&T_x^aT_h^bT_y^d\pi_1(F_2(m,s))\\
&\begin{aligned}=&\theta_{(p-m,q-s)}T_x^aT_h^b{T_y}^{n}\prod_{i=1}^{p-m}\{\prod_{i\in 2\mathbb{Z}}P_1(\frac{i}{2})\prod_{i\in 2\mathbb{Z}+1}Q_1(-\frac{i-1}{2})\}
\prod_{j=1}^{q-s-1}\prod_{i=0}^{p-m-1}\\&\{\prod_{i+j\in2\mathbb{Z}}Q_1(-\frac{i}{2}+jl)\prod_{i+j\in2\mathbb{Z}+1}P_1(\frac{i+1}{2}-jl)\}\end{aligned}\\
&=\theta_{(p-m,q-s)}T_x^aT_h^b{T_y}^{n}\prod_{j=0}^{q-s-1}\prod_{i=0}^{p-m-1}\{\prod_{i+j\in2\mathbb{Z}}Q_1(-\frac{i}{2}+jl)\prod_{i+j\in2\mathbb{Z}+1}P_1(\frac{i+1}{2}-jl)\}\\
&=\theta_{(p-m,q-s)}T_x^a\prod_{j=0}^{q-s-1}\prod_{i=n}^{d-1}\{\prod_{i+j\in2\mathbb{Z}}Q_1(-\frac{i}{2}+jl)\prod_{i+j\in2\mathbb{Z}+1}P_1(\frac{i+1}{2}-jl)\}T_h^b{T_y}^{n}\\
&\equiv \theta_{(p-m,q-s)}T_x^a\prod_{j=0}^{q-s-1}\prod_{i=n}^{d-1}\{\prod_{i+j\in2\mathbb{Z}}(T_h-\frac{i}{2}+jl)\prod_{i+j\in2\mathbb{Z}+1}(-\frac{i+1}{2}+jl)\}T_h^b{T_y}^{n}
~~(\mbox{mod}~T_xU(L_0)).
\end{align*}
Note that $-\frac{i+1}{2}+jl\ne 0$ for $i+j\in2\mathbb{Z}+1,0\leq j\leq q-s-1, n\leq i\leq d-1$, we have
\begin{equation*}
\begin{aligned}
&U(L_0)\pi_1(F_2(m,s))+T_xU(L_0)\\
=&T_xU(L_0)+\sum_{n=0}^{\infty}\mathbb{C}[T_h](\prod_{i+n+j\in2\mathbb{Z},0\leq j\leq q-s-1, 0\leq i\leq p-m-1}(T_h-\frac{i+n}{2}+jl)){T_y}^{n}.
\end{aligned}
\end{equation*}
Thus
\begin{equation*}
 \begin{aligned}
 &U(L_0)\pi_1(F_1(m,s))+U(L_0)\pi_1(F_2(m,s))+T_xU(L_0)\\
 \subset &T_xU(L_0)+U(L_0)T_y^m+\sum_{n=0}^{m-1}\mathbb{C}[T_h](\prod_{i+n+j\in2\mathbb{Z},0\leq j\leq q-s-1, 0\leq i\leq p-m-1}(T_h-\frac{i+n}{2}+jl)){T_y}^{n}.
 \end{aligned}
\end{equation*}
On the other hand,
since $\frac{i+d-1}{2}+jl\ne \frac{i^\prime+n}{2}-j^\prime l$ for any $i+j\in2\mathbb{Z}+1,1\leq j\leq s, 1\leq i\leq m,d\in\mathbb{Z}_+$ and
$i^\prime+j^\prime\in2\mathbb{Z},0\leq j^\prime\leq q-s-1, 0\leq i^\prime\leq p-m-1,n\in\mathbb{Z}_+$,
$\prod_{i+d+j\in2\mathbb{Z}+1,1\leq j\leq s, 1\leq i\leq m}(t-\frac{i+d-1}{2}-jl)$
and $\prod_{i+n+j\in2\mathbb{Z},0\leq j\leq q-s-1, 0\leq i\leq p-m-1}(t-\frac{i+n}{2}+jl)$
are relatively prime.
Then we have
\begin{equation*}
\mathbb{C}[T_h]T_y^{m+i}\subseteq U(L_0)\pi_1(F_1(m,s))+U(L_0)\pi_1(F_2(m,s))+T_xU(L_0)
\end{equation*}
for any $i\in\mathbb{Z}_+$. Thus
\begin{equation*}
 \begin{aligned}
 &U(L_0)\pi_1(F_1(m,s))+U(L_0)\pi_1(F_2(m,s))+T_xU(L_0)\\
\supset &T_xU(L_0)+U(L_0)T_y^m+\sum_{n=0}^{m-1}\mathbb{C}[T_h](\prod_{i+n+j\in2\mathbb{Z},0\leq j\leq q-s-1, 0\leq i\leq p-m-1}(T_h-\frac{i+n}{2}+jl)){T_y}^{n}.
 \end{aligned}
\end{equation*}
Set $t_1=T_h, t_2=T_y$.
Then by Proposition \ref{propavam} and Lemma \ref{lemfz},
we complete the proof.
\end{proof}

Now we apply the Zhu's algebras $A(L_{\widehat{\mathfrak{g}}}(\mathcal{l},0))$
and their bimodules $A(L(\mathcal{l},\mathcal{j}))$ which we obtain in Theorem \ref{thmav} and Theorem \ref{prop510} to calculate the fusion rules
among the irreducible ordinary modules of $(L_{\widehat{\mathfrak{g}}}(\mathcal{l},0),\omega_\xi)$ by Theorem \ref{thmam}.

\begin{prop}\label{thm47}
For admissible weights $\mathcal{j}_i=\frac{m_i-1}{2}-ls_i~(i=1,2)$,
the fusion rules are given as follows:
\begin{equation}
L(\mathcal{l},\mathcal{j}_1)\boxtimes L(\mathcal{l},\mathcal{j}_2)=
\sum_{n={\rm max}\{0,m_1+m_2-p\}}^{{\rm min}\{m_1-1,m_2-1\}}L(\mathcal{l},\mathcal{j}_1+\mathcal{j}_2-n)
\end{equation}
if $0\leq s_2\leq q-s_1-1$, and $L(\mathcal{l},\mathcal{j}_1)\boxtimes L(\mathcal{l},\mathcal{j}_2)=0$ otherwise.
\end{prop}
\begin{proof}
For any admissible weight $\mathcal{j}$, let $\mathbb{C}v_{\mathcal{j}}$
be the one-dimensional module for the Lie algebra $\mathbb{C}h$ such that $hv_{\mathcal{j}}=\mathcal{j}v_{\mathcal{j}}$.
From Theorem \ref{thmam}, we need to calculate the $A(L_{\widehat{\mathfrak{g}}}(\mathcal{l},0))$-module
$A(L(\mathcal{l},\mathcal{j}_1))\otimes_{A(L_{\widehat{\mathfrak{g}}}(\mathcal{l},0))}\mathbb{C}v_{\mathcal{j}_2}$.
From Theorem \ref{prop510}, we have
\begin{equation*}
A(L(\mathcal{l},\mathcal{j}_1))\otimes_{A(L_{\widehat{\mathfrak{g}}}(\mathcal{l},0))}\mathbb{C}v_{\mathcal{j}_2}\cong \mathbb{C}[t_1,t_2]/J,
\end{equation*}
where $J$ is the subspace of $\mathbb{C}[t_1,t_2]$ spanned by
\begin{equation*}
  \{\mathbb{C}[t_1,t_2](t_1-\mathcal{j}_2),\mathbb{C}[t_1,t_2]t_2^{m_1}+f_{\mathcal{j}_1,n}(\mathcal{j}_2,1)\mathbb{C}[t_1]t_2^n,n=0,1,\cdots,m_1-1\}.
\end{equation*}
If $\mathcal{j}_2$ does not satisfy the relation $0\leq s_2\leq q-s_1-1$, then
$$f_{\mathcal{j}_1,n}(\mathcal{j}_2,1)= (\prod_{i+n+j\in2\mathbb{Z},0\leq j\leq q-s_1-1, 0\leq i\leq p-m_1-1}(\mathcal{j}_2-\frac{i+n}{2}+jl))\ne0 $$
for $0\leq n\leq m_1-1$.
Thus $A(L(\mathcal{l},\mathcal{j}_1))\otimes_{A(L_{\widehat{\mathfrak{g}}}(\mathcal{l},0))}\mathbb{C}v_{\mathcal{j}_2}=0$,
so that all the corresponding fusion rules are zero.

Suppose $0\leq s_2\leq q-s_1-1$.
Note that $\mathbb{C}[t_1]t_2^n=0$ in $\mathbb{C}[t_1,t_2]/J$ if and only if $f_{\mathcal{j}_1,n}(\mathcal{j}_2,1)\ne0$.
Since $f_{\mathcal{j}_1,n}(\mathcal{j}_2,1)= (\prod_{i+n+j\in2\mathbb{Z},0\leq j\leq q-s_1-1, 0\leq i\leq p-m_1-1}(\mathcal{j}_2-\frac{i+n}{2}+jl))=0$
if and only if $\mathcal{j}_2-\frac{i+n}{2}+jl=0$ for some $i+n+j\in2\mathbb{Z},0\leq j\leq q-s_1-1, 0\leq i\leq p-m_1-1$.
This implies that $1\leq i+n+1\leq p-1$, by Remark \ref{remj} we have $i+n+1=m_2,j=s_2$,
that is $m_1+m_2-p\leq n\leq m_2-1$.
Therefore
\begin{equation}\label{eq000}
\mbox{max}\{0,m_1+m_2-p\}\leq n\leq \mbox{min}\{m_1-1,m_2-1\}.
\end{equation}
Then $\mathbb{C}[t_1]t_2^n\ne0$ in $\mathbb{C}[t_1,t_2]/J$ if and only if $n$ satisfies (\ref{eq000}).
Thus
\begin{equation*}
\mathbb{C}[t_1,t_2]/J\cong \bigoplus_{n={\rm max}\{0,m_1+m_2-p\}}^{{\rm min}\{m_1-1,m_2-1\}}\mathbb{C}t_2^n.
\end{equation*}
From Proposition \ref{propavam},
we have $t\ast t_2^n=(\mathcal{j}_1+\mathcal{j}_2-n)t_2^n$,
then we complete the proof.
\end{proof}

If $\mathcal{l}\in\mathbb{Z}_+$,
from Proposition \ref{propadm}, $q=1, p=2\mathcal{l}+3$, then for any admissible weights $\mathcal{j}_i=\frac{m_i-1}{2}-ls_i~(i=1,2)$,
we have $s_i=0, \mathcal{j}_i=\frac{m_i-1}{2}$, hence $0\leq s_2\leq q-s_1-1$ holds automatically.
Since $m_1+m_2-p=2\mathcal{j}_1+2\mathcal{j}_2-2\mathcal{l}-1$, then by Proposition \ref{thm47} we have the following corollary.

\begin{cor}\label{thm48}
Suppose $\mathcal{l}\in\mathbb{Z}_+$.
For any admissible weight $\mathcal{j}_1$ and $\mathcal{j}_2$, we have
\begin{equation}
L(\mathcal{l},\mathcal{j}_1)\boxtimes L(\mathcal{l},\mathcal{j}_2)
=\sum_{ \mathcal{j}=|\mathcal{j}_1-\mathcal{j}_2|,\mathcal{j}+\mathcal{j}_1+\mathcal{j}_2\leq 2\mathcal{l}+1}^{\mathcal{j}_1+\mathcal{j}_2}L(\mathcal{l},\mathcal{j}).
\end{equation}
\end{cor}
\begin{rem}
{\rm (i) Since $L^\prime(-1)=L(-1)$ the fusion rules among the weak modules
with respect two different vertex operator superalgebra structure of $L_{\widehat{\mathfrak{g}}}(\mathcal{l},0)$
are same. Thus the fusion rules obtained in the Proposition \ref{thm47} and Corollary \ref{thm48} are also those with respect to
the old vertex operator superalgebra structure.\\
(ii) The fusion rules among the irreducible relaxed highest-weight modules over $L_{\widehat{\mathfrak{g}}}(\mathcal{l},0)$
with admissible level
have been calculated in \cite{CKLR}
by viewing $L_{\widehat{\mathfrak{g}}}(\mathcal{l},0)$ as extensions of the tensor product of certain
simple Virasoro vertex operator algebra and simple affine vertex operator algebra associated to $sl_2$.}
\end{rem}


\begin{thebibliography}{99}
\bibitem{A}
T. Arakawa,
Rationality of admissible affine vertex algebras in the category $\mathcal{O}$,
{\em Duke Math. J.} {\bf 165} (2016), no. 1, 67-93.

\bibitem{AV}
T. Arakawa, J. Van Ekeren,
Modularity of relatively rational vertex algebras and fusion rules of principal affine $W$-algebras,
{\em Comm. Math. Phys.} {\bf 370} (2019), no. 1, 205-247.


\bibitem{AM}
D. Adamovi\'{c}, A. Milas,
Vertex operator algebras associated to modular invariant representations for $A_1^{(1)}$,
{\em Math. Res. Lett.} {\bf 2} (1995), no. 5, 563-575.




\bibitem{CGL}
T. Creutzig, N. Genra, A. Linshaw,
Category $\mathcal{O}$ for vertex algebras of $\mathfrak{osp}_{1|2n}$, 2022,
arXiv:math.RT/2203.08188.

\bibitem{CKLR}
T. Creutzig, S. Kanade, T. Liu, D. Ridout,
Cosets, characters and fusion for admissible-level $\mathfrak{osp}(1|2)$ minimal models,
{\em Nuclear Phys. B} {\bf 938} (2019), 22-55.


\bibitem{DLM1}
C. Dong, H. Li, G. Mason,
Regularity of rational vertex operator algebras,
{\em Adv. Math.} {\bf 132} (1997), no. 1, 148-166.


\bibitem{DLM2}
C. Dong, H. Li, G. Mason,
Vertex operator algebras associated to admissible representations of $\hat{sl}_2$,
{\em Comm. Math. Phys.} {\bf 184} (1997), no. 1, 65-93.


\bibitem{DZ2}
C. Dong, Z. Zhao,
Modularity of trace functions in orbifold theory for $\mathbb{Z}$-graded vertex operator superalgebras,
Moonshine: the first quarter century and beyond, 128-143,
London Math. Soc. Lecture Note Ser., 372, Cambridge Univ. Press, Cambridge, 2010.


\bibitem{DK}
A. De Sole, V. Kac,
Finite vs affine $W$-algebras,
{\em Jpn. J. Math.} {\bf 1} (2006), no. 1, 137-261.


\bibitem{V}
J. Van Ekeren,
Modular invariance for twisted modules over a vertex operator superalgebra,
{\em Comm. Math. Phys.} {\bf 322} (2013), no. 2, 333-371.

\bibitem{FHL}
I. Frenkel, Y. Huang, J. Lepowsky,
On axiomatic approaches to vertex operator algebras and modules,
{\em Mem. Am. Math. Soc.} {\bf 104} (1993).

\bibitem{FZ}
I. Frenkel, Y. Zhu,
Vertex operator algebras associated to representations of affine and Virasoro algebras,
{\em Duke Math. J.} {\bf 66} (1992), no. 1, 123-168.



\bibitem{GS}
M. Gorelik, V. Serganova,
Snowflake modules and Enright functor for Kac-Moody superalgebras,
{\em Algebra Number Theory} {\bf 16} (2022), no. 4, 839-879.

\bibitem{IK}
K. Iohara, Y. Koga,
Fusion algebras for $N=1$ superconformal field theories through coinvariants $\textrm{2}$: $\widehat{\mathfrak{osp}}(1|2)$-symmetry,
{\em J. reine angew. Math.} {\bf 531}(2001), 1-34.

\bibitem{K}
V. Kac, Lie superalgebras,
{\em Adv. Math.} {\bf 26} (1977), 8-96.

\bibitem{KK}
V. Kac, D. Kazhdan,
Highest weight representations for affine Lie algebras,
{\em Adv. Math.} {\bf 34} (1979), 97-108.

\bibitem{KW}
V. Kac, M. Wakimoto,
Modular invariant representations of infinite-dimensional Lie algebras and superalgebras,
{\em Proc. Nat. Acad. Sci. U.S.A.} {\bf 85} (1988), no. 14, 4956-4960.


\bibitem{KW2}
V. Kac, M. Wakimoto,
Classification of modular invariant representations of affine algebras,
Infinite-dimensional Lie algebras and groups (Luminy-Marseille, 1988), 138-177,
Adv. Ser. Math. Phys., 7, World Sci. Publ., Teaneck, NJ, 1989.

\bibitem{KWa}
V. Kac, W. Wang,
Vertex operator superalgebras and their representations,
Mathematical aspects of conformal and topological field theories and quantum groups (South Hadley, MA, 1992), 161-191,
Contemp. Math., 175, Amer. Math. Soc., Providence, RI, 1994.

\bibitem{L}
H. Li,
Local systems of vertex operators, vertex superalgebras and modules,
{\em J. Pure Appl. Algebra} {\bf 109} (1996), no.2, 143-195.

\bibitem{Li2}
H. Li,
Determining fusion rules by $A(V)$-modules and bimodules,
{\em J. Algebra} {\bf 212} (1999), no. 2, 515-556.

\bibitem{LL}
J. Lepowsky, H. Li,
Introduction to vertex operator algebras and their representations,
Progress in Math., Vol. 227, Birkh\"auser, Boston, 2004.


\bibitem{Lin}
X. Lin,
Rationality of affine vertex operator superalgebras with rational conformal weights,
{\em Comm. Math. Phys.} {\bf 402} (2023), no. 3, 2765-2790.

\bibitem{MFF}
F. Malikov, B. Feigin, D. Fuks,
Singular vectors in Verma modules over Kac-Moody algebras,
{\em J. Funct. Anal. Appl.} {\bf 20} (1986), no. 2, 25-37.




\bibitem{W}
S. Wood,
Admissible level $\mathfrak{osp}(1|2)$ minimal models and their relaxed highest weight modules,
{\em Transform. Groups} {\bf 25} (2020), no. 3, 887-943.

\end{thebibliography}
\end{document}